%% file: arxiv_main.tex
\documentclass[11pt]{article}
\usepackage{natbib}
\usepackage[margin=0.95in]{geometry}
\usepackage{graphicx}
\usepackage{amsmath,amssymb,amsfonts,amsthm,mathrsfs}
\usepackage{color}
\usepackage[small,bf]{caption}
\usepackage{makecell}
\newcommand{\mylabel}[2]{#2\def\@currentlabel{#2}\label{#1}}
\makeatother
\usepackage[hypertexnames=false]{hyperref}
\hypersetup{
  colorlinks=true,
  linkcolor=blue,
  filecolor=blue,
  citecolor = blue,      
  urlcolor=cyan,
}

\usepackage[utf8]{inputenc} % allow utf-8 input
\usepackage[T1]{fontenc}    % use 8-bit T1 fonts
\usepackage{hyperref}       % hyperlinks
\usepackage{url}            % simple URL typesetting
\usepackage{booktabs}       % professional-quality tables
\usepackage{amsfonts}       % blackboard math symbols
\usepackage{nicefrac}       % compact symbols for 1/2, etc.
\usepackage{microtype}      % microtypography
\usepackage{xcolor}         % colors
\usepackage{multirow}
\usepackage{quoting}

\usepackage{algorithm}
\usepackage{algorithmic}
\usepackage{verbatim}

\usepackage{xspace}
\usepackage{mathtools}
\usepackage[nameinlink,capitalize]{cleveref}

\usepackage{enumitem}

\include{commands}

\include{header}

\title{Population Dynamics Control with Partial Observations}

\author{
  Zhou Lu\thanks{
  Princeton University. 
  \texttt{zhoul@princeton.edu}.}
  \and Y.Jennifer Sun\thanks{
  Princeton University.
  \texttt{ys7849@princeton.edu}.}
  \and Zhiyu Zhang\thanks{
  Harvard University.
  \texttt{zhiyuz@seas.harvard.edu}.}
}

\begin{document}

\maketitle

\begin{abstract}%
We study the problem of controlling population dynamics—a class of linear dynamical systems evolving on the probability simplex—from the perspective of online non-stochastic control. While \cite{golowich2024online} analyzed the fully observable setting, we focus on the more realistic, partially observable case, where only a low-dimensional representation of the state is accessible.

In classical non-stochastic control, inputs are set as linear combinations of past disturbances. However, under partial observations, disturbances cannot be directly computed. To address this,  
\cite{simchowitz2020improper} proposed to construct oblivious signals, which are counterfactual observations with zero control, as a substitute. This raises several challenges in our setting: (1) how to construct oblivious signals under simplex constraints, where zero control is infeasible; (2) how to design a sufficiently expressive convex controller parameterization tailored to these signals; and (3) how to enforce the simplex constraint on control when projections may break the convexity of cost functions.

Our main contribution is a new controller that achieves the optimal $\tilde{O}(\sqrt{T})$ regret with respect to a natural class of mixing linear dynamic controllers. To tackle these challenges, we construct signals based on hypothetical observations under a constant control adapted to the simplex domain, and introduce a new controller parameterization that approximates general control policies linear in non-oblivious observations. Furthermore, we employ a novel convex extension surrogate loss, inspired by \cite{lattimore2024bandit}, to bypass the projection-induced convexity issue.

\end{abstract}

\input{intro}
\input{minkowski}

\input{po}

\input{nonmarkov}
\input{discussion}

\newpage

\bibliographystyle{plainnat}
\bibliography{main}

\newpage
\appendix
\input{appendix}

\newpage
\input{extension}

\end{document}

%% file: commands.tex
\newcommand{\nc}{\newcommand}
\nc{\DMO}{\DeclareMathOperator}
\nc\m[2]{m_{#1}(#2)}
\nc{\BR}{\mathbb{R}}
\nc{\BN}{\mathbb{N}}
\nc{\BZ}{\mathbb{Z}}
\nc{\ep}{\varepsilon}
\nc{\ra}{\rightarrow}
\DMO{\KL}{KL}
\DMO{\Unif}{Unif}
\nc{\tr}{\top}

\nc{\MD}{\mathcal{D}}
\nc{\MN}{\mathcal{N}}
\nc{\SD}{\mathscr{D}}
\nc{\la}{\lambda}
\nc{\MR}{\mathcal{R}}
\nc{\MZ}{\mathcal{Z}}
\nc{\MU}{\mathcal{U}}
\nc{\MP}{\mathcal{P}}
\nc{\poly}{\mathrm{poly}}
\DMO{\treesum}{TreeSum}
\DMO{\lapsum}{LapSum}
\DMO{\checksum}{CheckSum}
\nc{\MDts}{\MD_{\treesum}}
\nc{\MDls}{\MD_{\lapsum}}
\nc{\MDcs}{\MD_{\checksum}}
\nc{\MC}{\mathcal{C}}
\nc{\MI}{\mathcal{I}}
\nc{\MT}{\mathcal{T}}
\nc{\MS}{\mathcal{S}}
\nc{\MM}{\mathcal{M}}
\nc{\MX}{\mathcal{X}}
\nc{\MY}{\mathcal{Y}}
\nc{\MA}{\mathcal{A}}
\nc{\MB}{\mathcal{B}}
\nc{\MJ}{\mathcal{J}}
\nc{\MF}{\mathcal{F}}
\nc{\MQ}{\mathcal{Q}}
\nc{\ML}{\mathcal{L}}
\nc{\p}{p}

\nc{\E}{\mathbb{E}}
\nc{\tablesize}{s}
\DMO{\Hist}{hist}
\nc{\hist}{\mathrm{hist}}

\nc{\ba}{\mathbf{a}}
\nc{\bx}{\mathbf{x}}
\nc{\by}{\mathbf{y}}
\nc{\bz}{\mathbf{z}}

\DMO{\sr}{sr}
\DMO{\Med}{Med}
\DMO{\Ber}{Ber}
\DMO{\Bin}{Bin}
\DMO{\Had}{Had}

\nc{\ME}{\mathcal{E}}
\DMO{\View}{View}
\nc{\B}{B}
\nc{\M}{M}
\nc{\ha}{\alpha}
\nc{\hk}{k}
\DMO{\pre}{pre}
\nc{\MH}{\mathcal{H}}
\DMO{\FO}{FO}
\DMO{\CM}{CM}
\DMO{\hb}{\beta}
\nc{\MW}{\mathcal{W}}
\nc{\wnull}{\mathbf{w}^\circ}
\nc{\wshift}{\mathbf{w}}

\makeatletter
\newtheorem*{rep@theorem}{\rep@title}
\newcommand{\newreptheorem}[2]{%
\newenvironment{rep#1}[1]{%
 \def\rep@title{#2 \ref{##1}}%
 \begin{rep@theorem}}%
 {\end{rep@theorem}}}
\makeatother

\definecolor{mygreen}{RGB}{0, 200, 0}
\definecolor{myred}{RGB}{100,200,0}

\usepackage{bbm}
\usepackage{mathabx}
\nc{\One}{\mathbbm{1}}
\renewcommand{\^}[1]{^{[#1]}}
\newcommand{\GPC}{\texttt{GPC}\xspace}

\nc{\st}{\star}
\nc{\tvnorm}[2]{\left\| {#1} - {#2} \right\|_1}

\nc{\tnorm}[1]{\left\| {#1} \right\|_1}
\nc{\tmix}{t^{\mathsf{mix}}}

\nc{\norm}[1]{\left\|{#1}\right\|}
\nc{\infonenorm}[1]{\left\| {#1} \right\|_{1 \to 1}}
\nc{\oneonenorm}[1]{\left\|{#1}\right\|_{1\to 1}}
\nc{\twoonenorm}[1]{\left\| {#1} \right\|_{2, 1}}
\nc{\Xgpc}[1][d,H,a_0,\alphaub]{\MX_{#1}}
\nc{\Rgpc}[1][d,H]{R_{#1}}
\nc{\gpcnorm}[2][d,H]{\left\| {#2} \right\|_{#1}}

\nc{\MK}{\mathcal{K}}

\nc{\Ksim}{\MK^\triangle}

\nc{\Pins}{\Pi^{\mathsf{ns}}}
\nc{\bP}{\mathbf{P}}
\nc{\PolReg}{\textbf{PolReg}}

\nc{\aw}{\bar{w}}
\nc{\ax}{\bar{x}}
\nc{\au}{\bar{u}}
\nc{\sdist}[2][d]{\Delta_{ {#2}}^{#1}}
\nc{\dist}[2][d]{\Delta_{{#2}}^{#1}}
\nc{\sstoch}[2][d]{\mathbb{S}^{#1}_{#2}}
\nc{\Cr}[1]{\mathbb{A}_{#1}}
\nc{\cmp}{\mathrm{c}}
\mathchardef\mhyphen="2D
\nc{\GPCS}{\mathtt{GPC\mhyphen Simplex}\xspace}
\nc{\GPCPOS}{\mathtt{GPC\mhyphen PO\mhyphen Simplex}\xspace}
\nc{\NMGPCPOS}{\mathtt{Non\mhyphen Markov\mhyphen GPC\mhyphen PO\mhyphen Simplex}\xspace}
\nc{\Alg}{\mathtt{Alg}}

\nc{\BS}{\mathbb{S}}
\nc{\OGD}{\mathtt{OGD}}
\DMO{\Proj}{Proj}

\nc{\alphalb}{\underline{\alpha}}
\nc{\alphaub}{\overline{\alpha}}

\nc{\alb}{a_0}
\nc{\aub}{a_1}
\nc{\grad}{\nabla}
\nc{\rng}{\rangle}
\nc{\lng}{\langle}
\nc{\vep}{\varepsilon}

\newcommand{\breg}{\mathbf{regret}}

%% file: header.tex
\newcommand{\eps}{\ensuremath{\varepsilon}}
\newcommand{\R}{\ensuremath{\mathbb R}}

\theoremstyle{plain}
\newtheorem{theorem}{Theorem}
\newtheorem{lemma}[theorem]{Lemma}

\newtheorem{observation}[theorem]{Observation}
\newtheorem{assumption}{Assumption}

\newtheorem*{theorem*}{Theorem}
\newtheorem*{lemma*}{Lemma}
\newtheorem*{corollary*}{Corollary}
\newtheorem*{proposition*}{Proposition}
\newtheorem*{claim*}{Claim}
\newtheorem*{fact*}{Fact}
\newtheorem*{observation*}{Observation}
\newtheorem*{assumption*}{Assumption}

\theoremstyle{definition}
\newtheorem{definition}[theorem]{Definition}

\newtheorem*{definition*}{Definition}
\newtheorem*{remark*}{Remark}
\newtheorem*{example*}{Example}

\newcommand{\ignore}[1]{}

\def\K{\mathcal{K}}

\def\E{{\mathbb{E}}}

\usepackage{color}
\usepackage{nicefrac}

%% file: intro.tex
\section{Introduction}

Online non-stochastic control \citep{agarwal2019online,hazan2020nonstochastic,hazan2022introduction} is an emerging framework at the intersection of online learning and control theory. By leveraging techniques from \emph{Online Convex Optimization} (OCO), it enables the design of \emph{no-regret} controllers that go beyond traditional assumptions, such as i.i.d. disturbances and time-invariant cost functions, in optimal control. In this work, we extend this framework to the control of \emph{population dynamics}, a class of \emph{Linear Dynamical Systems} (LDS) evolving on the probability simplex. These systems, which arise in epidemic modeling and evolutionary game theory, are not \emph{strongly stabilizable}—that is, there exists no state-feedback linear controller that can drive the system asymptotically to zero in the absence of disturbances. This violates a key assumption underpinning much of non-stochastic control theory. 

\cite{golowich2024online} recently studied this problem in the \emph{fully observable} setting, where the controller has access to the exact high-dimensional state (i.e., the full population distribution). However, in many practical applications, full observability is unrealistic due to measurement constraints. A more practical setting assumes that only a low-dimensional representation of the state is available—its dimensionality reflecting the number of effective measurements that can be made, as exemplified by the celebrated Kalman filter \citep{kalman1960new} and the rich literature on inverse problems. From a technical perspective, this \emph{partially observable} setting presents a fundamental challenge: the standard approach in non-stochastic control—designing controls as linear combinations of past disturbances \citep{hazan2022introduction}—becomes infeasible since past disturbances can no longer be directly computed.

This paper focuses on controlling partially observable population dynamics and introduces new techniques to handle the intricate interplay between partial observability and the simplex constraint. Our contributions include (i) constructing tailored signals as alternatives to incomputable disturbances, (ii) designing a novel convex controller parameterization that approximates control policies linear in non-oblivious observations, and (iii) introducing a convex extension surrogate loss to ensure that projections are applied only during control execution, thereby preserving the convexity in parameter updates.

\paragraph{Problem setting.} 
Let $\Delta^{d}:=\left\{p\in\BR^d: p\geq 0,\norm{p}_1=1\right\}$ be the probability simplex, and $\BS^{d_1,d_2}:=\left\{M\in\BR^{d_1 \times d_2}:Mx \in\Delta^{d_1} \textrm{ for all } x \in \Delta^{d_2}\right\}$ be the set of matrices mapping $\Delta^{d_2}$ to $\Delta^{d_1}$. Given the state-control dimension $d_x$ and measurement dimension $d_y$, we define a \emph{partially observable simplex LDS} as follows:

\begin{definition} [Partially observable simplex LDS]
\label{def:po-simplex-lds}
Let $A, B\in\mathbb{S}^{d_x}$ and $C\in\mathbb{S}^{d_y,d_x}$ be the system matrices, $x_1\in\Delta^{d_x}$ be some initial state, and $\alpha\in [0,1]$ be the fixed control power. A \emph{partially observable simplex LDS} is defined as a dynamical system such that at each time step $t\geq 1$, given the control input $u_t\in\Delta^{d_x}$, the disturbance strength $\gamma_t\in[0,1]$ and the disturbance $w_t\in\Delta^{d_x}$, its state $x_t\in\Delta^{d_x}$ and output $y_t\in\Delta^{d_y}$ evolve according to
\begin{align*}
x_{t+1}&=(1-\gamma_t)[(1-\alpha)Ax_t+\alpha Bu_t]+\gamma_tw_t, \\
y_{t+1} &= Cx_{t+1}.
\end{align*}
\end{definition}

The goal is to design controllers that, given knowledge of $A, B, C, \alpha$, and $\gamma_t$, can operate without direct access to the internal state. The input-output protocol is the following.
\begin{enumerate}
\item At the beginning of the $t$-th step, Nature picks $\gamma_t\in[0,1]$ and $w_t\in\Delta^{d_x}$ obliviously. 
\item The controller observes the measurement $y_t\in\Delta^{d_y}$ and the disturbance strength $\gamma_t\in[0,1]$,\footnote{See \citep[Appendix~B]{golowich2024online} for a justification of observing $\gamma_t$. Roughly speaking, it is often feasible to observe the total size of the population in each step, from which $\gamma_t$ can be computed.} and then picks a control $u_t\in\Delta^{d_x}$. 
\item Nature reveals a convex cost function $c_t:\Delta^{d_y}\times \Delta^{d_x}\rightarrow\BR_+$ to the controller, and the controller incurs the cost $c_t(y_t,u_t)$.
\end{enumerate}

For a given sequence of nature’s choices, let $y^\pi_t$ and $u^\pi_t$ denote the observations and controls under an arbitrary controller $\pi$. Our objective is to design a controller $\MA$ such that regardless of nature's decisions, the \emph{regret} of $\MA$ with respect to a suitable comparator class $\Pi$ satisfies 
\begin{equation*}
\breg_\MA(\Pi,T):=\sum_{t=1}^Tc_t\left(y_t^{\MA},u_t^{\MA}\right)-\min_{\pi^*\in \Pi}\sum_{t=1}^T c_t\left(y_t^{\pi^*},u_t^{\pi^*}\right)=o(T).
\end{equation*}

\paragraph{Technical challenges.} This problem presents unique challenges due to the interplay between population dynamics and partial observability:
\begin{itemize}
\item \textbf{Memory effects in non-stochastic control:} Unlike OCO, the cost function $c_t(y_t, u_t)$ depends not only on $u_t$ but on the entire history $(u_1, \dots, u_{t-1})$ through $x_t$. Many works assume strong stabilizability\footnote{There exists a
“default” controller such that $x_t$ decays to zero exponentially fast without disturbances.} to mitigate this issue, reducing the problem to an OCO-like setting with logarithmic memory. However, population dynamics are only \emph{marginally stable}: states neither explode nor decay to zero.
\item \textbf{Partial observability and convexity:} Classical non-stochastic control strategies parameterize $u_t$ as a linear combination of past disturbances, ensuring convexity. This relies on computing $w_{t-1}$ from state observations, which is impossible under partial observability.
\end{itemize}

While \cite{golowich2024online} and \cite{simchowitz2020improper} addressed these issues separately, their combination introduces new difficulties. \cite{simchowitz2020improper} proposed using counterfactual observations under zero control (termed \emph{nature’s $y$}) instead of disturbances, but in population dynamics, zero control is infeasible. This raises two key questions: (i) how to construct suitable signals under the simplex constraint and (ii) how to ensure control feasibility without violating convexity. In particular, if we directly apply the approach of \cite{simchowitz2020improper} to population dynamics, the resulting $u_t$ could easily violate the simplex constraint, and a naive projection would introduce nonconvexity to the subsequent optimization problem. 

\paragraph{Results and techniques.} 
We propose two controllers achieving the optimal $\tilde{O}(\sqrt{T})$ regret against different comparator classes. A simpler version competes with linear \emph{Markov} controllers of the form $u_t = K y_t$, while a more advanced version handles \emph{Linear Dynamical Controllers} (LDCs), a richer class with internal state tracking. Our main innovations include:
\begin{enumerate}
    \item \textbf{Convex extension surrogate loss:} We make a novel use of a convex extension surrogate loss, coupled with the \emph{Minkowski projection} \citep{mhammedi2022efficient}, to bypass the simplex constraint on the output of the controller. While similar ideas exist in OCO \citep{cutkosky2018black}, our specific use case in partially observable non-stochastic control is more challenging.
    \item \textbf{New signal construction:} Instead of counterfactual observations under zero control, we construct signals based on a constant control policy $u_t = K_0 y_t$, where $K_0$ is a stochastic matrix. This enables a new convex controller parameterization tailored to the simplex constraint, which requires new techniques beyond the existing unconstrained analysis \citep{simchowitz2020improper}.
\end{enumerate}

\paragraph{Related works.} There has been a surge of interest in online non-stochastic control since the seminal work of \citet{agarwal2019online}; see \citep{hazan2022introduction} for a recent survey. Beyond the full-information, fully observed setting originally studied by \cite{agarwal2019online}, a series of works extended this framework to settings with limited feedback, such as bandit feedback \citep{cassel2020bandit,sun2024optimal,suggala2024second,sun2024tight} and partial observations \citep{simchowitz2020improper,yan2023online}. A natural requirement of the partially observed setting is the comparator class being LDCs, as such controllers encompass the optimal $H_{\infty}$ control law which is the hallmark of classical optimal control theory \citep{bacsar2008h}. However, due to the difficulty of analyzing LDCs, technical refinements since \citep{simchowitz2020improper} have been scarce. 

Going beyond strong stabilizability, \citet{golowich2024online} recently studied the non-stochastic control of population dynamics in the fully observed setting. Readers are referred to their Section~1 and 5 for the background of such system models, as well as application-oriented case studies.

\paragraph{Organization.} 
\cref{sec:min} introduces Minkowski projection and its applications. \cref{sec:po-prelim-and-approx} presents a controller for Markov comparators, while \cref{sec:non-markov} generalizes to LDCs. Proofs are deferred to the appendix due to space limit. 

%% file: minkowski.tex
\section{Minkowski Extension and Projection}\label{sec:min}

To address constraint issues in population dynamics, we introduce a convex extension technique. A similar approach has been explored in OCO (Online Convex Optimization), as presented in \citep[Section~4]{cutkosky2018black}. However, we extend this technique to the more complex setting of non-stochastic control.

\paragraph{Motivation.} Our control learning class outputs a ``raw control'' $\tilde{u}_t(M) \in \mathbb{R}^{d_x}$, which is linear in the learning parameters $M$. In population dynamics, this raw control may not necessarily satisfy the constraint $\Delta^{d_x}$. A naive approach would be to project $\tilde{u}_t$ onto $\Delta^{d_x}$, but this breaks the convexity: while $c_t(y_t, u_t)$ is convex in $u_t$, it becomes non-convex in $M$ after projection.

To resolve this, we separate control \emph{updating} from \emph{execution}. Specifically, we use a convex \emph{surrogate loss} $e_t$ defined on $\mathbb{R}^{d_x}$ rather than $c_t$, allowing learning updates to occur in an unconstrained space. Projection is applied only at execution time, preserving convexity in parameter updates. This mirrors the philosophy of mirror descent, where updates occur in a dual space before mapping back to the primal space.

Our approach leverages a convex extension technique based on the \emph{Minkowski functional} (also known as the \emph{gauge function}). While the Minkowski functional is a classical concept in convex analysis \citep{rockafellar1970convex}, it has recently been used in projection-free online learning \citep{mhammedi2022efficient} and bandit convex optimization \citep{fokkema2024online}. Inspired by \cite{lattimore2024bandit}, we construct a convex extension function using this tool.

\paragraph{Definitions.} Throughout the following definitions, we let $\K\subset\mathbb{R}^d$ be a convex set satisfying $0\in\K$. We define the Minkowski functional as follows:

\begin{definition} [Minkowski functional]
The Minkowski functional $\pi_\mathcal{K}:\mathbb{R}^d\rightarrow\mathbb{R}$ associated with the set $\MK$ is defined as $\pi_{\K}(x)=\inf\{t\ge 1: x\in t\mathcal{K}\}$.
\end{definition}

Intuitively, $\pi_{\mathcal{K}}(x)$ is the smallest positive scaling of $\mathcal{K}$ required to enclose $x$. Since we are interested in cases where $x \notin \mathcal{K}$, we enforce $\pi_{\mathcal{K}}(x) \geq 1$. By construction, $\pi_{\mathcal{K}}$ is convex, and $\pi_{\mathcal{K}}(x) = 1$ if and only if $x \in \mathcal{K}$.
We now define two key constructs: the \emph{Minkowski projection} and the \emph{Minkowski extension}\footnote{We note that the definition of the Minkowski function is not unique, as shown in \citep[Section~3.7]{lattimore2024bandit}. Here we use a simple form not considered there, but sufficient for our purpose.}.

\begin{definition} [Minkowski projection]
The Minkowski projection to the set $\MK$ is a function mapping $\mathbb{R}^d\rightarrow\K$, given by $\Proj_{\Delta^{d_x}}^{MP}(x)= \frac{x}{\pi_{\K}(x)}$. 
\end{definition}

\begin{definition}[Minkowski extension]\label{def:minkowski-extension}
For any convex function $f:\K\rightarrow\mathbb{R}_+$, its Minkowski extension $e_f:\mathbb{R}^d\rightarrow\mathbb{R}_+$ with respect to the set $\MK$ is defined as
\begin{align}
\label{eq:extension}
e_f(x)=\pi_{\mathcal{K}}(x)\cdot f\left(\frac{x}{\pi_{\mathcal{K}}(x)}\right).
\end{align}
Here the notation $\MK$ is omitted on the LHS for conciseness.
\end{definition}

\subsection{Key Properties}\label{sec:property-minkowski}

The Minkowski extension satisfies several useful properties, making it well-suited as a surrogate loss function.

\begin{lemma} [Validity as surrogate loss]
\label{lem:extension-property}
Let $\K\subset\mathbb{R}^d$ be a convex set satisfying $0\in\K$, and $f:\K\rightarrow\mathbb{R}_+$ be a convex function. 
The Minkowski extension $e_f:\mathbb{R}^d\rightarrow\mathbb{R}_+$ from Definition~\ref{def:minkowski-extension} satisfies the following properties:
\begin{enumerate}
    \item $e_f(x)=f(x)$ for all $x\in \mathcal{K}$.
    \item $e_f(x)\ge f(\frac{x}{\pi_{\mathcal{K}}(x)})$ for all $x\in\R^d$.
    \item $e_f$ is convex on $\mathbb{R}^d$.
\end{enumerate}
\end{lemma}

Another important aspect is the Lipschitzness of the Minkowski extension, such that using it as a surrogate loss is tractable from the perspective of optimization. We will first require a curvature property on $\mathcal{K}$.

\begin{definition} [$\kappa$-aspherity]
\label{def:aspherity}
Given $d\in\mathbb{N}$, a convex set $\K\subset\mathbb{R}^d$ is said to have $\kappa$-aspherity for some $\kappa\ge 1$ if $\exists r, R>0$ (without loss of generality, assume that $r\le 1\le R$) such that $\kappa=\frac{R}{r}$ and
\begin{align*}
r\mathbb{B}^d\subset\K\subset R\mathbb{B}^d, 
\end{align*}
where $\mathbb{B}^d=\{x\in\mathbb{R}^d\mid \|x\|_2\le 1\}$ denotes the unit ball in $\mathbb{R}^d$.  
\end{definition}

\begin{lemma} [Lipschitzness of extension function]
\label{lem:extension-lip}
Let $\K\subset\mathbb{R}^d$ be a convex set with $\kappa$-aspherity, and $f:\K\rightarrow\mathbb{R}_+$ be a convex, $L_f$-Lipschitz function (w.r.t. $\|\cdot\|$ satisfying $\|\cdot\|\ge \|\cdot\|_2$) that attains $0$ on $\K$. Let $D_{\K}$ be the diameter of $\K$ in $\|\cdot\|$.  Then the extension fucntion $e_f$ is $(4DL_f\kappa^4+\frac{D_{\K}L_f}{r})$-Lipschitz on $D\mathbb{B}^d$, $\forall D>0$. When $R\ge 1$, $e_f$ is $(4DL_f\kappa^4+D_{\K}L_f\kappa)$-Lipschitz w.r.t. $\|\cdot\|_2$.
\end{lemma}

First, $e_f$ is only Lipschitz around the set $\MK$, which is fine since the underlying unconstrained algorithm stays around $\MK$. Second, the Lipschitz constant increases with the aspherity of $\MK$. This motivates our reparameterization on the control set described next.

\subsection{Application to Control}\label{sec:control-using-minkowski}
To apply Minkowski extensions in non-stochastic control, we must adapt the simplex constraint because $\Delta^{d_x}$ doesn't contain 0 and has $\kappa=0$ in $\mathbb{R}^{d_x}$. The following

 linear bijective map $\phi$ from $\Delta^{d_x}$ to $\mathbb{R}^{d_x-1}$ re-centers the simplex in a low-dimensional space.
\begin{align}
\label{eq:phi-func}
    \forall u\in\Delta^{d_x}, \phi(u)= (u_1,\dots,u_{d_x-1})-\frac{\mathbf{1}_{d_x-1}}{2(d_x-1)}.
\end{align}

\begin{definition} [Control set]
\label{def:observation-control-set}
Given $d_x\in\mathbb{N}$, $d_x\ge 2$, the feasible control class is given by  $\Delta^{d_x}$.
Additionally, define a lower-dimensional class $\K\subset\mathbb{R}^{d_x-1}$ as the following:
\begin{align}
\label{eq:pair-constrained-set}
    \K=\left\{v-\frac{\mathbf{1}_{d_x-1}}{2(d_x-1)}:v\in\mathbb{R}^{d_x-1},v\ge 0, \|v\|_1\le 1\right\}.
\end{align}
\end{definition}

\begin{lemma}[Properties of the control set]
\label{cor:lip-control}
Let $\K$ be defined in Definition~\ref{def:observation-control-set}, and let $\pi_{\mathcal{K}}$ be the associated Minkowski functional. Then $\pi_{\mathcal{K}}$ satisfies the following properties:
\begin{enumerate}
\item $\pi_{\mathcal{K}}$ is $2d_x$-Lipschitz on $\mathbb{R}^{d_x-1}$. 
\item The extension $e_f$ of a convex $L_f$-Lipschitz (w.r.t. $\|\cdot\|_1$) function $f:\K\rightarrow\mathbb{R}_+$ is $O(DL_fd_x^4)$-Lipschitz (w.r.t. $\|\cdot\|_2$) over $D\mathbb{B}^{d_x-1}$ if $f$ is bounded by $B$ on $D\mathbb{B}^{d_x-1}$ for $D\ge 1$.
\end{enumerate}
\end{lemma}

%% file: po.tex
\section{Controlling under Partial Observability}
\label{sec:po-prelim-and-approx}

In this section, we give an efficient online controller $\GPCPOS$ for partially observable simplex LDS, which has provable regret guarantees against a particular class of Markov linear policies. 
We begin with a standard assumption on the cost functions analogous to Assumption 1 in \citep{golowich2024online}:

\begin{assumption}[Lipschitz, decomposable convex cost]
\label{asm:po-convex-loss}
The cost functions $c_t:\Delta^{d_y}\times \Delta^{d_x} \to \BR_+$ take the form \footnote{$c_t$ is assumed to be separated for variables $y,u$, which is standard (e.g. in LQR, \citep{simchowitz2020naive}).} 
\begin{align*}
c_t(y,u)=c_t^y(y)+c_t^u(u),
\end{align*}
where $c_t^y:\Delta^{d_y}\rightarrow\mathbb{R}_+, c_t^u:\Delta^{d_x}\rightarrow\mathbb{R}_+$ are convex and $L$-Lipschitz in the following sense: for all $y,y'\in\Delta^{d_y}$ and $u,u'\in\Delta^{d_x}$, we have
\begin{align*}
|c_t^y(y)-c_t^y(y)|\le L\cdot \|y-y'\|_1, \ \ \ |c_t^u(u)-c_t^u(u')|\le L\cdot \|u-u'\|_1.
\end{align*}
Without loss of generality, assume that $L\ge 1$ and $\exists (y,u)\in\Delta^{d_y}\times\Delta^{d_u}$ such that $c_t(y,u)=0$. 
\end{assumption}

We define the class of policies against which we prove regret bounds. These are valid linear policies that “mix” the state of the system, akin to the fully observable setting.

\begin{definition} [Comparator policy class]
\label{def:po-comparator}
Let $\ML$ be a partially observable simplex LDS with transition matrices $A,B \in \mathbb{S}^{d_x}$, observation matrix $C \in \mathbb{S}^{d_y, d_x}$, and control power $\alpha \in [0,1]$. Given a matrix $K\in\mathbb{S}^{d_x, d_y}$, define \[\mathbb{A}_K^{\alpha}:=(1-\alpha)A+\alpha BKC\in\mathbb{S}^{d_x}.\]
We define the comparator policy class $\Ksim_\tau = \Ksim_\tau(\ML)$ as the set of linear, time-invariant policies $u_t=K y_t$ such that $K\in\mathbb{S}^{d_x,d_y}$ satisfies $\tmix(\mathbb{A}_{K}^{\alpha})\le \tau$. Here
\begin{align*}
t^{\mathrm{mix}}(\mathbb{A}_{K}^{\alpha})=\min_{t\in\mathbb{N}}\left\{t: \quad \forall t'\ge t, \sup_{p\in\Delta^{d_x}}\|(\mathbb{A}_{K}^{\alpha})^t(p)-\pi(\mathbb{A}_{K}^{\alpha})\|_1\le \frac{1}{4}\right\}\le \tau,
\end{align*}
where $\pi(\mathbb{A}_{K}^{\alpha})$ denotes the stationary of $\mathbb{A}_{K}^{\alpha}$.
\end{definition}

For any $K \in \BS^{d_x, d_y}$ and partially observable simplex LDS $\ML$, we let $(y_t(K))_t$ and $(u_t(K))_t$ denote the counterfactual sequences of observations and controls that would have been obtained by following the time-invariant, linear policy $u_t=K y_t$ in $\ML$.
Our result for controlling partially observable simplex with Markov linear comparator policy class is stated below
(we omit the dependence on the system $\ML$, since it is always clear from context).

\begin{theorem} [PO regret: Markov comparator policy class]
\label{thm:po-main} 
Let $d_x,d_y, T\in\BN$ with $d_y\leq d_x$. %Let $\ML=(A,B, C, x_1,(\gamma_t)_{t\in\BN},(w_t)_{t\in\BN},(c_t)_{t\in\BN}, \alpha)$ be a partially observable simplex LDS with cost functions $(c_t)_t$ satisfying \cref{asm:po-convex-loss} for some $L > 0$. 
Set $H := 2\tau \lceil \log(512d_x^{7/2}LT^2) \rceil$, $\eta=1/(Ld_x^{4.5}H^{2}\sqrt{T})$. Then for any $\tau > 0$, the iterates  $(y_t, u_t)$ of $\GPCPOS$ (\cref{alg:gpc-po}) with input $(A, B, C, \tau, H,T)$ have the following policy regret guarantee with respect to $\Ksim_\tau(\ML)$: \begin{align*}
\sum_{t=1}^T c_t(y_t,u_t)-\min_{K^{\star}\in\Ksim_\tau(\ML)}\sum_{t=1}^T c_t(y_t(K^{\star}),u_t(K^{\star}))\le \tilde{O}\left(L\tau^{4}d_x^{6.5}\sqrt{T}\right),
\end{align*}
where $\tilde{O}(\cdot)$ hides all universal constants and poly-logarithmic dependence in $T$. 
\end{theorem}

\subsection{Our Algorithm}
\label{sec:po-alg}

The algorithm $\GPCPOS$ (\cref{alg:gpc-po}) is in broad strokes similar to $\GPCS$ \citep{golowich2024online}; it uses online convex optimization over a class of policies that is expressive enough to approximate any policy in the comparator class.

The general idea is the following: although the comparator class consists of only linear Markov policies $u_t=K y_t$, we can't directly learn this $K$ by OCO because $y_t$ is non-oblivious. Instead, the non-stochastic control approach reparameterizes the control $u_t$ as a linear combinations of computable oblivious signals $o_t$ (Line~\ref{line:control-po-utot} in Algorithm~\ref{alg:gpc-po}) with small memory $H$. Then because $c_t$ is convex in $u_t$, it's also convex in the control parameters $M$, reducing the original control problem to a tractable OCO problem over $M$. As long as the comparator class is a subset of the learning class, any regret bound of this OCO problem directly transfers to the control comparator class of interest.

To state the algorithm precisely, we need to define the optimization domain $\MM$ (analogous to $\Xgpc[d,H,a_0,\alphaub]$ in $\GPCS$) and the correspondence between points in the domain and policies. First, the optimization domain $\MM_{d_x,d_y,H}$ is defined below:

\begin{algorithm}[t]
\caption{$\GPCPOS$}
\label{alg:gpc-po}
\begin{algorithmic}[1]
\REQUIRE Transition matrices $A,B \in \BS^{d_x}$, observation matrix $C \in \BS^{d_y,d_x}$, horizon parameter $H \in \BN$, control parameter $\alpha\in[0,1]$, step size $\eta>0$.  Let $\K$ be defined as in \cref{eq:pair-constrained-set}, and $\pi_{\mathcal{K}}$ be the Minkowski functional associated with it.  
Denote as $\phi^{-1}$ the inverse of $\phi$ when restricted to $\Delta^{d_x}$.

\STATE Initialize $M_1^{[0:H]}\in \MM := \MM_{d_x,d_y,H}$ (Definition~\ref{def:learning-class}). Choose any $K_0\in\mathbb{S}^{d_x,d_y}$. \label{line:choose-k0}
\STATE Observe the initial observation $y_1$.
\FOR{$t=1,\cdots,T$}
    \STATE Compute $o_t := y_t-\sum_{i=1}^{t-1}\bar{\lambda}_{t,i}C[(1-\alpha)A]^{i-1}\alpha B(u_{t-i}-K_0 o_{t-i})$. \label{line:compute-signal}
    \STATE Let $\tilde{u}_t=M_t(o_{t:t-H})=M_{t}^{[0]}o_t+\sum_{i=1}^H \bar{\lambda}_{t,i}M_{t}^{[i]}o_{t-i}$.\label{line:control-po-utot}
    \STATE Choose control $u_t=\phi^{-1}(\phi(\tilde{u}_t)/\pi_{\mathcal{K}}(\phi(\tilde{u}_t))$.\label{line:control-po}
    \STATE Receive cost $c_t(y_t,u_t)=c_t^y(y_t)+c_t^u(u_t)$, and observe $y_{t+1}$ and $\gamma_t$.
    \STATE Define surrogate loss function 
\begin{align*}
e_t(M) := c_t^y(y_t(M; K_0))+ \pi_{\mathcal{K}}(\phi(\tilde{u}_t(M;K_0))) \cdot c_t^u\left(\phi^{-1}\left(\frac{\phi(\tilde{u}_t(M;K_0))}{\pi_{\mathcal{K}}(\phi(\tilde{u}_t(M;K_0))}\right)\right).
\end{align*} \label{line:et-loss}
    \STATE Set $M_{t+1} := \Pi_{\mathcal{M}}^{\|\cdot\|_2}\left[M_t-\eta \cdot \partial e_t(M_t)\right]$, where $\Pi_{\mathcal{M}}^{\|\cdot\|_2}$ denotes the $\ell_2$ projection and $\partial e_t(M_t)$ is the subgradient of $e_t$ evaluated at $M_t$. \label{line:po-update}
\ENDFOR{}
\end{algorithmic}
\end{algorithm}

\begin{definition} [Learning policy class] 
\label{def:learning-class} 
Let $d_x,d_y,H \in \BN$ and let $\BS_{0}^{d_x,d_y}$ be the set of matrices $K\in[-1,1]^{d_x\times d_y}$ such that $\sum_{i=1}^{d_x} K_{ij} = 0$ for all $j \in [d_y]$. Then let $\MM_{d_x,d_y,H}$ be the set of tuples $M^{[0:H]} = (M^{[0]},\dots,M^{[H]})$ defined as follows:
\begin{align}
\label{eq:M-class}
\MM_{d_x,d_y,H} := \{M^{[0:H]}\mid M^{[0]}\in\BS^{d_x,d_y} \text{ and } M^{[i]}\in\BS_{0}^{d_x,d_y} \text{ for all }  i \in [H]\}.
\end{align}
\end{definition}

For dimension parameters $d_x,d_y \in \BN$ and history length $H \in \BN$, the domain $\MM_{d_x,d_y,H}$ parametrizes a class of policies $\{\pi^M: M \in \MM_{d_x,d_y,H}\}$ where $\pi^M$ is defined as follows:

\begin{definition} [Learning policy]
\label{def:learning-policy}
Fix $d_x,d_y \in \BN$ and $M = M^{[0:H]} \in \MM_{d_x,d_y,H}$. We define a policy $\pi^M$ that for any time $t$, given $\gamma_{t-1},\dots,\gamma_{t-H}$ and \emph{signals} $o_t,\dots,o_{t-H} \in \Delta^{d_y}$, produces the control
\begin{align*}
\pi^M_t(o_{t:t-H}) := \phi^{-1}\left(\Proj_{\Delta^{d_x}}^{MP}\left[\phi\left(M^{[0]}o_t+\alpha\sum_{i=1}^{H}\bar{\lambda}_{t,i}M^{[i]}o_{t-i}\right)\right]\right)
\end{align*}
where $\bar\lambda_{t,i}$ is defined for $t \in [T]$ and $i \in \BN^+$, 
\begin{align}
\lambda_{t,i} := \gamma_{t-i} \cdot \prod_{j=1}^{i-1} (1-\gamma_{t-j}), \qquad \bar \lambda_{t,i} := \prod_{j=1}^{i} (1-\gamma_{t-j})\label{eq:define-lambdas}, \qquad \lambda_{t,0} := 1 - \sum_{i=1}^H \lambda_{t,i}.
\end{align}
%and $\Proj_{\Delta^{d_x}}^{MP}: \BR^{d_x} \to \Delta^{d_x}$ denotes the Minkowski projection operator onto $\Delta^{d_x}$,\zz{let's define this in Sec 2?} and $\phi$ denotes the linear transformation \cref{eq:phi-func}. 
\end{definition}

\iffalse
where $u_t(M; K_0)$ is the control at time $t$ according to $M$ with signals $y_{t-H:t}(K_0)$, given by
\begin{align*}
u_t(M;K_0)=u_t(M^{[0:H]};K_0)=\Pi_{\Delta^{d_x}}^{\|\cdot\|_2}\left[M^{[0]}y_t(K_0)+\sum_{i=1}^{H}\bar{\lambda}_{t,i}M^{[i]}y_{t-i}(K_0)\right],
\end{align*}

\fi

In the above definitions, note that (1) $\mathcal{M}_{d_x,d_y,H}$ is convex, and (2) the projection step in the definition of $\pi^M$ is necessary (in order for $\pi^M(o_{t-1:t-H})$ to be a valid control, i.e. in $\Delta^{d_x}$) since the vector before projection has indices summing to $1$ but is not necessarily entry-wise non-negative.

\paragraph{Loss Function.}
For notational convenience, for any $M \in \MM_{d_x,d_y,H}$, $K_0 \in \BS^{d_x,d_y}$ and $H \in \BN$, we write
\begin{align*}
\tilde{u}_t(M;K_0) &:= M^{[0]}y_t(K_0) + \alpha \sum_{i=1}^H \bar\lambda_{t,i} M^{[i]} y_{t-i}(K_0), \\ 
y_t(M;K_0) &:= \sum_{i=1}^{t-1}\bar\lambda_{t,i} C((1-\alpha)A)^{i-1}\alpha B u_{t-i}(M;K_0) + \sum_{i=1}^t \lambda_{t,i} ((1-\alpha)A)^{i-1} w_{t-i}.
\end{align*}

By unfolding, $y_t(M;K_0)$ is exactly the counterfactual observation at time $t$ if control $u_s(M;K_0) = \pi^M(y_{s:s-H}(K_0))$ had been chosen at time $s$, for each $1 \leq s \leq t$. We then define the surrogate loss function at time $t$ to be
\begin{align}
\label{eq:et-def}
e_t(M) := c_t^y(y_t(M; K_0))+ \pi_{\mathcal{K}}(\phi(\tilde{u}_t(M;K_0))) \cdot c_t^u(\phi^{-1}(\frac{\phi(\tilde{u}_t(M;K_0))}{\pi_{\mathcal{K}}(\phi(\tilde{u}_t(M;K_0))})).
\end{align}

\begin{observation}[Extension properties]
\label{obs:ext}
$e_t$ defined in \cref{eq:et-def} is convex and satisfies $e_t(M)\ge c_t(y_t(M;K_0),u_t(M;K_0))$, $\forall M\in\MM_{d_x,d_y,H}$.
\end{observation}

%In \cref{sec:approx-po}, we show that every policy in the comparator class (Definition~\ref{def:po-comparator}) can be approximated by $\pi^M$ for some $M \in \MM_{d_x,d_y,H}$. We bound the memory mismatch error in \cref{sec:mem-mismatch-po}, and conclude the proof in \cref{sec:po-proof}.% and revisit the online gradient descent regret guarantee in \cref{sec:gd-prelim}.

\subsection{Proof Overview}
Similar to $\GPCS$, $\GPCPOS$ runs gradient steps on the counterfactual loss, and thus the regret decomposition follows similarly to that of $\GPCS$ (Algorithm 1 in \cite{golowich2024online}). The analysis breaks down into three main ingredients: 
\begin{enumerate}
    \item We show that the counterfactual observations $y_t(K_0)$ under some ``default'' linear policy $K_0$, can be exactly computed to be the signals $o_t$ in Line~\ref{line:compute-signal} of \cref{alg:gpc-po}. In particular, through an unrolling technique, we will show that $y_t(K_0)$ can be constructed as adding a calibration term based on the difference between controls $u_{t-i} - K_0 y_{t-i}(K_0)$, to the raw observation $y_t$:
    \begin{align*} y_t(u_{1:t-1}) - \sum_{i=1}^{t-1} \bar\lambda_{t,i}C((1-\alpha)A)^{i-1}\alpha B(u_{t-i} - K_0 y_{t-i}(K_0)) = y_t(K_0).\end{align*}
   
    \item Next, we show the expressivity of the learning policy class: every policy $K^*$ in the comparator class (Definition~\ref{def:po-comparator}) can be approximated by $\pi^M$ for some $M \in \MM_{d_x,d_y,H}$. This is done by constructing $M$ such that $u_t(M;K_0)$ and $u_t(K^*)$ are very close (implying $y_t(K^*)$ also close to $y_t(M;K_0)$, then costs are close by Lipschitzness), which breaks into a few steps. First, we explicitly construct $M$ such that $u_t(K^*)$ is exactly equal to the control with full memory $\hat{u}_t(M;K_0)$, defined as below 
\begin{align*}
\hat{u}_{t}(M;K_0):=M^{[0]}y_t(K_0)+\alpha\sum_{i=1}^{t-1}\bar{\lambda}_{t,i}M^{[i]}y_{t-i}(K_0).
\end{align*}
    Then we show a low truncation error between $\hat{u}_t(M;K_0)$ and $\tilde{u}_t(M;K_0)$ (Line~\ref{line:control-po} in Algorithm~\ref{alg:gpc-po}) which has only $H$ history dependence: roughly speaking, we show the difference term $\left\|\sum_{i=H+1}^{t-1}\bar{\lambda}_{t,i}M^{[i]}y_{t-i}(K_0)\right\|_1$ can be written as $\sum_{i=H+1}^{t-1} O\left(\left\|A_{\text{mix}}^i(v_1-v_2)\right\|_1\right)$ for some mixing matrix $A_{\text{mix}}$ and probability vectors $v_1,v_2$, here $A_{\text{mix}}^iv_1, A_{\text{mix}}^iv_2$ are very close due to the mixing property. Finally we show $\tilde{u}_t(M;K_0)$ is close to its projection.

    \item Then, we bound the memory mismatch error: the difference between the true loss and the would-be loss when all controls in history are replaced by $M_t$. Here a new treatment on the discussion about mixing time is required, due to the new parameterization.
\end{enumerate}
Finally, after proving regularity properties of $e_t$ required by the previous steps, we can put everything together and conclude the proof. Below is the regret decomposition, where the extension mismatch error is non-positive by definition, and the comparator mismatch error can be easily bounded by constant.
\begin{align*}
&\sum_{t=1}^T c_t(y_t,u_t)-\min_{K^{\star}\in \Ksim_\tau(\ML)}\sum_{t=1}^T c_t(y_t(K^{\star}),u_t(K^{\star}))\\
&\le \underbrace{\sum_{t=1}^T c_t(y_t,u_t)-\sum_{t=1}^T c_t(y_t(M_t;K_0),u_t(M_t;K_0))}_{\text{(memory mismatch error)}} + \underbrace{\sum_{t=1}^T c_t(y_t(M_t;K_0),u_t(M_t;K_0))- \sum_{t=1}^Te_t(M_t)}_{\text{(extension mismatch error})} \\
&\quad + \underbrace{\sum_{t=1}^Te_t(M_t)-\min_{M\in\mathcal{M}}\sum_{t=1}^T e_t(M)}_{\text{(regret on extension)}} + \underbrace{\min_{M\in\mathcal{M}}\sum_{t=1}^T e_t(M)- \min_{K^{\star}\in \Ksim_\tau(\ML)}\sum_{t=1}^T c_t(y_t(K^{\star}),u_t(K^{\star}))}_{\text{(comparator mismatch error)}}
\end{align*}

%% file: nonmarkov.tex
\section{Extending to Non-Markov Policy}\label{sec:non-markov}

The results in the previous section hold for Markov linear policies, where the control 
$u_t$ depends only on the most recent observation $y_t$. However, in partially observable control problems, it is well-known that Markov policies may not be expressive enough to capture the optimal policy, even in the i.i.d. noise setting \citep{bacsar2008h}. Thus, it is crucial to study whether optimal regret bounds can be achieved for a broader class of non-Markov policies.

\subsection{Motivation: The Need for a Richer Policy Class}
Prior work by \cite{simchowitz2020improper} studied control with partial observations under $\ell_2$ constraints. They noted that a richer policy class is required to capture the $H_{\infty}$ control policy in the i.i.d. noise setting. To address this, they introduced the class of stabilizing linear dynamic controller (LDC):
\begin{align*}
s_{t+1}=A_{\pi}s_t+B_{\pi}y_t, \quad u_t=C_{\pi}s_t+D_{\pi}y_t. 
\end{align*}

Here, the control is computed as a linear combination of the current observation $y_t$ and a hidden state $s_t$, which evolves according to a linear system driven by past observations. This structure enables the policy to incorporate the full history of observations rather than relying solely on 
$y_t$.

The LDC structure originates from the optimal solution for Linear Quadratic Gaussian (LQG) control with partial observations \citep{bacsar2008h}. In LQG, the key idea is to estimate an approximate state $\tilde{x}_t$ using a Kalman filter, then compute the optimal control $u_t$, as if $\tilde{x}_t$ were the true state $x_t$. The estimation evolves as 
$$\tilde{x}_{t+1}=A \tilde{x}_t+Bu_t+D(y_{t+1}-C(A \tilde{x}_t+Bu_t)),$$
where $D(y_{t+1}-C(A \tilde{x}_t+Bu_t))$ is the correction term to the original system based on the current observation $y_{t+1}$. Since $u_t$ is linear in $\tilde{x}_t$, $\tilde{x}_{t+1}$ is simply linear in $\tilde{x}_t$ and $y_{t+1}$, LDCs are expressive enough to capture classical LQG control with i.i.d. Gaussian noise.

\subsection{Extending to Population Dynamics}
For population dynamics, the optimal linear control policy under full observability was derived by \cite{golowich2024online}. To extend this to partial observability, we need an analogous comparator class to LDCs. Similarly we need to simulate $x_t$ by $s_t$ in the stochastic setting, the main difference lies in the original evolution of $x_t$:
$$
x_{t+1}=(1-\gamma_t)(Ax_t+Bu_t)+\gamma_t w_t.
$$
In contrast to classical control, where noise is Gaussian with zero mean, population dynamics involve disturbances $w_t$ over the simplex $\Delta^{d_x}$ with a nonzero mean $v$. A natural choice is to set
$$
s_{t+1}=(1-\gamma_t)(As_t+Bu_t)+\gamma_t v.
$$
Since $\gamma_t$ and $w_t$ are independent, $E[x_t]=E[s_t]$ holds by the linearity of expectation when $s_1=x_1$. Thus, this system is expressive enough to capture optimal population control with stochastic noise. For simplicity, we assume $v=\frac{1}{d} \mathbf{1}_d$ which is the most natural prior, and our analysis can be generalized to for any $v\in \Delta_d$.

\begin{definition}[simplex-LDC]
\label{def:po-comparator-ldc}
Given $d_x, d_y\in\mathbb{N}$, a simplex linear dynamical controller (simplex-LDC) is a LDC $\pi=(A_{\pi}, B_{\pi}, C_{\pi}, D_{\pi})$ of appropriate dimensions with the additional assumptions that
\begin{enumerate}
\item $s_1=\frac{1}{d_x}\mathbf{1}_{d_x}\in\Delta^{d_x}$.
\item $\forall s\in\Delta^{d_x}, y\in\Delta^{d_y}$, there holds that the next inner state and control evolved according to the simplex LDC are in the simplex, i.e.
\begin{align*}
s_{t+1}=(1-\gamma_t)((1-\alpha)A_{\pi}s_t+\alpha B_{\pi}y_t)+\gamma_t\cdot\frac{1}{d_x}\mathbf{1}_{d_x}\in\Delta^{d_x}, \quad u_t=C_{\pi}s_t+D_{\pi}y_t\in\Delta^{d_x}. 
\end{align*} 
\end{enumerate}
\end{definition}

Unlike standard control settings where states decay to zero without noise, population dynamics maintain a unit $\ell_1$-norm due to the simplex constraint. Therefore, we replace strong stabilizability with a mixing time assumption.

\begin{definition}[$\tau$-mixing simplex-LDC]
\label{def:mixing-simplex-ldc}
Given $\tau>0$ and a partially observable simplex LDS 
%$$\mathcal{L}=(A, B, C, x_1, (\gamma_t)_{t\in\mathbb{N}}, (w_t)_{t\in\BN},(c_t)_{t\in\BN}, \alpha)$$ 
with state and observation dimensions $d_x, d_y\in\mathbb{N}$, we denote as $\Pi_{\tau}(\ML)$ the class of $\tau$-mixing simplex-LDCs as any simplex-LDC $\pi=(A_{\pi}\in\mathbb{R}^{d_x\times d_x}, B_{\pi}\in\mathbb{R}^{d_x\times d_y}, C_{\pi}\in\mathbb{R}^{d_x\times d_x}, D_{\pi}\in\mathbb{R}^{d_x\times d_y}, s_1\in\mathbb{R}^{d_x})$ satisfying that for
\begin{align}
\label{eq:ldc-transition-matrix}
A_{\pi,\mathrm{cl}}=\begin{bmatrix}
(1-\alpha)A+\alpha BD_{\pi}C & \alpha BC_{\pi} \\
 B_{\pi}C & A_{\pi},
\end{bmatrix}
\end{align}
where $A_{\pi,\mathrm{cl}}:\Delta^{d_x}\times\Delta^{d_x}\rightarrow \Delta^{d_x}\times\Delta^{d_x}$, there exists unique $x\in\Delta^{d_x}$, $s\in\Delta^{d_x}$ such that 
\begin{align*}
t^{\mathrm{mix}}(A_{\pi,\mathrm{cl}})=\min_{t\in\mathbb{N}}\left\{t: \quad \forall t'\ge t, \sup_{p\in\Delta^{d_x},q\in\Delta^{d_x}}\|A_{\pi,\mathrm{cl}}^t(p,q)-(x,s)\|_1\le \frac{1}{4}\right\}\le \tau,
\end{align*}
where $(p,q), (x,s)\in\mathbb{R}^{d_x+d_x}$ denote the concatenated vectors of $p, q$ and $x,s$, respectively.
\end{definition}

In order to approximate the set of $\tau$-mixing simplex-LDC as in Definition~\ref{def:mixing-simplex-ldc}, we consider a slightly modified convex learning class than the one in Definition~\ref{def:learning-class} used in the Markov setting. 

\begin{definition} [Learning policy class for $\tau$-mixing simplex-LDC]
\label{def:learning-class-ldc}
Let $d_x,d_y,H \in \BN$ and let $\BS_{0}^{d_x,d_y}$ be set of matrices $K\in[-1,1]^{d_x\times d_y}$ such that $\sum_{i=1}^{d_x} K_{ij} = 0$ for all $j \in [d_y]$. Then let $\MM_{d_x,d_y,H}^{+}$ be the set of tuples ${M^{[0:H]}} = (M^{[0]},\dots,M^{[H]})$ defined as follows:
\begin{align}
\label{eq:M-class-nonmarkov}
\MM_{d_x,d_y,H}^{+} := \{M^{[0:H]}\mid M^{[0]}\in\BS^{d_x,d_y}\otimes \BS^{d_x,d_x} \text{ and } M^{[i]}\in\BS_{0}^{d_x,d_y}\otimes \BS_{0}^{d_x,d_x} \text{ for all }  i \in [H]\}.
\end{align}
\end{definition}
First, we observe that $\MM_{d_x,d_y,H}^{+}$ defined in Definition~\ref{def:learning-class-ldc} is convex. We will show that $\MM_{d_x,d_y,H}^{+}$ approximates the class of $\tau$--mixing simplex-LDC well. Therefore, a regret bound with respect to $\MM_{d_x,d_y,H}^{+}$ translates to the regret bound with respect to the class of $\tau$--mixing simplex-LDC. 

\subsection{Our Algorithm and Result}

With the learning policy class defined, our new controller algorithm (deferred to appendix) is a modified version of Algorithm~\ref{alg:gpc-po} to accommodate the new class of learning policies. Essentially, we raise the dimension of the signals to incorporate the uniform prior.  

For any $\pi \in \Pi_{\tau}(\ML)$ and partially observable simplex LDS $\ML$, we let $(y_t(\pi))_t$ and $(u_t(\pi))_t$ denote the counterfactual sequences of observations and controls that would have been obtained by playing the policy $\pi$ from the beginning of the time. We can now state our main result for controlling partially observable simplex LDS, with respect to the class of $\tau$-mixing simplex-LDC.

\begin{theorem} [PO regret ($\tau$-mixing simplex-LDC)]
\label{thm:po-main-ldc-simplex} 
Let $d_x,d_y, T\in\BN$ with $d_y\leq d_x$. %Let $\ML=(A,B, C, x_1,(\gamma_t)_{t\in\BN},(w_t)_{t\in\BN},(c_t)_{t\in\BN}, \alpha)$ be a partially observable simplex LDS with cost functions $(c_t)_t$ satisfying \cref{asm:po-convex-loss} for some $L > 0$. 
Set $H=\lceil 2\tau\lceil \log\frac{4096\tau d_x^{7/2} L T^2}{\epsilon}\rceil \rceil$, $\eta=1/(LH^{2}d_x^{6}\sqrt{T})$. Then for any $\tau > 0$, there exists an algorithm with the following policy regret guarantee with respect to $\Pi_{\tau}(\ML)$: \begin{align*}
\sum_{t=1}^T c_t(y_t,u_t)-\min_{\pi\in\Pi_{\tau}(\ML)}\sum_{t=1}^T c_t(y_t(\pi),u_t(\pi))\le \tilde{O}\left(L\tau^{4}d_x^{8}\sqrt{T}\right),
\end{align*}
where $\tilde{O}(\cdot)$ hides all universal constants and poly-logarithmic dependence in $T$. 
\end{theorem}

\subsection{New Technical Ingredients} Similarly to the previous setting, Theorem~\ref{thm:po-main-ldc-simplex} depends on several building blocks. 
First, we will show that the convex learning policy class (Definition~\ref{def:learning-class-ldc}) used in our algorithm approximates $\Pi_{\tau}(\ML)$ well. To incorporate the structure of population dynamics, we design a new LDC analogue
\begin{align*}
\begin{bmatrix}
x_{t+1}(\pi) \\
s_{t+1}(\pi)
\end{bmatrix}
&=(1-\gamma_t)A_{\pi,\mathrm{cl}}\begin{bmatrix}
x_t(\pi)\\
s_t(\pi)
\end{bmatrix}+\gamma_t \begin{bmatrix}
w_t\\
\frac{1}{d_x}\mathbf{1}_{d_x}
\end{bmatrix}, \\
\begin{bmatrix}
y_{t}(\pi) \\
u_{t}(\pi)
\end{bmatrix}
&=\begin{bmatrix}
C & 0\\
D_{\pi}C & C_{\pi}
\end{bmatrix}\begin{bmatrix}
x_t(\pi)\\
s_t(\pi)
\end{bmatrix},
\end{align*}
coupled with a change in the control $\tilde{u}_t$ by appending the prior below $o_t$. The additional $\frac{1}{d_x}\mathbf{1}_{d_x}$ term is crucial to proving the desired equivalence, such that after careful decomposition and rewriting of the control, we can use the important matrix identity $X^n=Y^n+\sum_{i=1}^n X^{i-1}(X-Y)Y^{n-i}$ as a bridge.

Next, we will show that the memory mismatch error is small. In this step, because our mixing assumption differs from that in the previous section, a fine-grained analysis is conducted on lower bounding $\alpha$ when the mixing time of $A$ is large. A key step is to bound the mixing time of $A$ by the mixing time of $A_{\pi,\mathrm{cl}}=\begin{bmatrix}
(1-\alpha)A+\alpha BD_{\pi}C & \alpha BC_{\pi} \\
 B_{\pi}C & A_{\pi}
\end{bmatrix}.$ Because they have different dimensions, we need an auxiliary matrix $A_{compare}=\begin{bmatrix}
A & 0 \\
 B_{\pi}C & A_{\pi}
\end{bmatrix}$ as a bridge. This requires a new analysis along with a slightly worse dependence on $d_x$, which comes from the $\ell_1$ norm estimation of a key term, which is no longer in $\Delta^{d_x}$ but lives in $\Delta^{d_x}-\Delta^{d_x}$.

%% file: discussion.tex
\section{Conclusion}
In this work, we investigate the control of population dynamics in a partially observable setting. Despite its practical significance, this problem presents several nontrivial technical challenges due to the interplay between simplex constraints and partial observability.

We develop controller algorithms with $\tilde{O}(\sqrt{T})$ regret that perform effectively against both linear Markov controllers and a broader class of non-Markov linear dynamic controllers. Our results build upon the frameworks of \cite{golowich2024online, simchowitz2020improper}, incorporating key algorithmic and analytical innovations to address the unique challenges of the problem.

%% file: appendix.tex
\tableofcontents

\newpage

\section{Omitted Proofs from Section \ref{sec:min}}

\subsection{Proof of Lemma~\ref{lem:extension-property}}
\begin{proof} The first two properties follow immediately from construction. 
    We only need to prove the third property, and we do this by showing that the directional derivative satisfies $\partial_x e_f(x)\ge 0, \forall x\notin \mathcal{K}$.
    
    For any two points $x_1,x_2$, denote $y_1=\frac{x_1}{\pi_{\mathcal{K}}(x_1)}, y_2=\frac{x_2}{\pi_{\mathcal{K}}(x_2)}$, we need to prove
    \begin{align}
    \label{eq:minkowski-convex-inequality}
    \pi_{\mathcal{K}}(x_1)f(y_1)+\pi_{\mathcal{K}}(x_2)f(y_2)\ge 2\pi_{\mathcal{K}}\left(\frac{x_1+x_2}{2}\right)f\left(\frac{x_1+x_2}{2\pi_{\mathcal{K}}\left(\frac{x_1+x_2}{2}\right)}\right).
    \end{align}
    We notice that
    $$
    \frac{x_1+x_2}{2\pi_{\mathcal{K}}(\frac{x_1+x_2}{2})}=ay_1+by_2,
    $$
    where 
    \begin{align*}
        a=\frac{\pi_{\mathcal{K}}(x_1)}{2\pi_{\mathcal{K}}\left(\frac{x_1+x_2}{2}\right)}, \quad b=\frac{\pi_{\mathcal{K}}(x_2)}{2\pi_{\mathcal{K}}(\frac{x_1+x_2}{2})}.
    \end{align*}
    By the convexity of $f$, the right hand side of \cref{eq:minkowski-convex-inequality} can be bounded by
    \begin{align*}
    2\pi_{\mathcal{K}}\left(\frac{x_1+x_2}{2}\right)f(ay_1+by_2)&\le 2\pi_{\mathcal{K}}\left(\frac{x_1+x_2}{2}\right)\frac{af(y_1)+bf(y_2)}{a+b}\\
    &=\frac{4\pi_{\mathcal{K}}^2\left(\frac{x_1+x_2}{2}\right)}{\pi_{\mathcal{K}}(x_1)+\pi_{\mathcal{K}}(x_2)}[af(y_1)+bf(y_2)],
    \end{align*}
    while the left hand side of \cref{eq:minkowski-convex-inequality} is just
    $$
    2\pi_{\mathcal{K}}\left(\frac{x_1+x_2}{2}\right)[af(y_1)+bf(y_2)].
    $$
    It suffice to prove 
    $$
    \frac{4\pi_{\mathcal{K}}^2\left(\frac{x_1+x_2}{2}\right)}{\pi_{\mathcal{K}}(x_1)+\pi_{\mathcal{K}}(x_2)}\le 2\pi_{\mathcal{K}}\left(\frac{x_1+x_2}{2}\right),
    $$
    which is equivalent to 
    $$
    \pi_{\mathcal{K}}(x_1)+\pi_{\mathcal{K}}(x_2)\ge 2\pi_{\mathcal{K}}\left(\frac{x_1+x_2}{2}\right),
    $$
    by the convexity of $\pi_{\mathcal{K}}$. Because both $\pi_{\mathcal{K}}, f$ are continuous and $f$ is non-negative, mid-point convexity implies convexity.
\end{proof}

\subsection{Proof of Lemma~\ref{lem:extension-lip}}

\begin{proof}
To analyze the curvature of Minkowski projections and extensions, we introduce the concept of non-expanding functions.

\begin{definition}[Non-expanding functions]
Given $d_1,d_2\in\mathbb{N}$. Let $\K\subseteq \mathbb{R}^{d_1}$.
Consider a function $f:\K\rightarrow\mathbb{R}^{d_2}$. Let $\|\cdot\|_a, \|\cdot\|_b$ be any two well-defined norms on $\mathbb{R}^{d_1}$, $\mathbb{R}^{d_2}$, respectively. $\forall L>0$, we say that $f$ is $L$-non-expanding w.r.t. the norm tuple $(\|\cdot\|_a, \|\cdot\|_b)$ if $\|f(x_1)-f(x_2)\|_b\le L\|x_1-x_2\|_a$, $\forall x_1,x_2\in\K$. 
\end{definition}

Lemma~\ref{thm:proj-lip} bounds the Lipschitzness of Minkowski projection onto an aspherical convex set. 

\begin{lemma} [Lipschitzness of Minkowski projection]
\label{thm:proj-lip}
Consider a convex set $\mathcal{K}\subset\mathbb{R}^d$ that satisfies $\kappa$-aspherity ($\kappa\ge 1$ by Definition~\ref{def:aspherity}) and its associated Minkowski functional $\pi_{\mathcal{K}}$.  We have that the Minkowski projection function $x\mapsto \frac{x}{\pi(x)}$ is $4\kappa^3$-non-expanding w.r.t. $(\|\cdot\|_2,\|\cdot\|_2)$ on $\mathbb{R}^d$, i.e. $\forall x_1,x_2\in\mathbb{R}^d$,
\begin{align*}
\left\|\frac{x_1}{\pi_{\mathcal{K}}(x_1)}-\frac{x_2}{\pi_{\mathcal{K}}(x_2)}\right\|_2\le 4\kappa^3\cdot\|x_1-x_2\|_2. 
\end{align*}
\end{lemma}

Using this, we obtain a Lipschitz bound for the Minkowski extension. Fix $u,v\in D\mathbb{B}^d$. By Lipschitz condition on $f$ and Lemma~\ref{thm:proj-lip}, we have that
\begin{align*}
\left|f\left(\frac{u}{\pi_{\mathcal{K}}(u)}\right)-f\left(\frac{v}{\pi_{\mathcal{K}}(v)}\right)\right|\le 4L_f\kappa^3\|u-v\|_2. 
\end{align*}
In addition, by \cite{lu2023projection}, we have that $\pi_{\mathcal{K}}$ is $\frac{1}{r}$-Lipschitz.
As a result,
\begin{align*}
|e(u)-e(v)|&=\left|\pi_{\mathcal{K}}(u)f\left(\frac{u}{\pi_{\mathcal{K}}(u)}\right)-\pi_{\mathcal{K}}(v)f\left(\frac{v}{\pi_{\mathcal{K}}(v)}\right)\right|\\
&\le \left|\pi_{\mathcal{K}}(u)f\left(\frac{u}{\pi_{\mathcal{K}}(u)}\right)-\pi_{\mathcal{K}}(u)f\left(\frac{v}{\pi_{\mathcal{K}}(v)}\right)\right|+\left|\pi_{\mathcal{K}}(u)f\left(\frac{v}{\pi_{\mathcal{K}}(v)}\right)-\pi_{\mathcal{K}}(v)f\left(\frac{v}{\pi_{\mathcal{K}}(v)}\right)\right|\\
&\le 4|\pi_{\mathcal{K}}(u)|\cdot L_f\kappa^3\|u-v\|_2+\left|f\left(\frac{v}{\pi_{\mathcal{K}}(v)}\right)\right|\cdot \frac{1}{r}\|u-v\|_2\\
&\le 4 DL_f\kappa^4\|u-v\|_2+D_{\K}L_f\cdot \frac{1}{r}\|u-v\|_2\\
&\le (4DL_f\kappa^4+\frac{D_{\K}L_f}{r})\|u-v\|_2, 
\end{align*}
where the second inequality uses the Lipschitz condition of $f$ and $\pi_{\mathcal{K}}$; the third inequality uses the fact that $\frac{u}{D\kappa}\in \K$. 
\end{proof}

\subsection{Proof of Lemma~\ref{thm:proj-lip}}

\begin{proof} Let $\kappa=\frac{R}{r}$ such that $r\mathbb{B}^d\subset\K\subset R\mathbb{B}^{d}$. We discuss three cases of $x_1,x_2$.

\paragraph{Case 1, $x_1,x_2\in \mathcal{K}$:} in this case $\pi_{\mathcal{K}}(x_1)=\pi_{\mathcal{K}}(x_2)$, and clearly $\|\frac{x_1}{\pi_{\mathcal{K}}(x_1)}-\frac{x_2}{\pi_{\mathcal{K}}(x_2)}\|_2=\|x_1-x_2\|_2$.

\paragraph{Case 2, $x_1,x_2\notin \mathcal{K}$:} first we prove the following useful lemma.

\begin{lemma}
\label{lem lip}
    A function $f$ is $L$-non-expanding, if and only if there exists some constant $C>0$, for any $x_1,x_2$ with $\|x_1-x_2\|_2\le C$, we have that
    $$
    \|f(x_1)-f(x_2)\|_2\le L\|x_1-x_2\|_2.
    $$
\end{lemma} 
\begin{proof}
    The only if direction is straightforward. For the if direction, for any $x_1,x_2$ with distance $\|x_1-x_2\|_2=N>0$, denote $n=\lfloor \frac{N}{C}\rfloor+1$, and denote $y_i=\frac{i}{n}x_1+\frac{n-i}{n}x_2$, for $i\in [n-1]$. Then $\|x_2-y_1\|_2, \|x_1-y_{n-1}\|_2, \|y_j-y_{j+1}\|_2$ are all bounded by $C$ for any $j\in [n-2]$. As a result, $\|f(x_2)-f(y_1)\|_2, \|f(x_1)-f(y_{n-1})\|_2, \|f(y_j)-f(y_{j+1})\|_2$ are all bounded by $\frac{NL}{n}$ for any $j\in [n-2]$. Therefore,
    \begin{align*}
         \|f(x_1)-f(x_2)\|_2&\le \|f(x_2)-f(y_1)\|_2+ \|f(x_1)-f(y_{n-1})\|_2+\sum_{j=1}^{n-2} \|f(y_j)-f(y_{j+1})\|_2 \\
         &\le NL =L\|x_1-x_2\|_2.
    \end{align*}
\end{proof}

Intuitively, Lemma~\ref{lem lip} tells us that we only need to consider $x_1,x_2$ with a small distance for proving the theorem. In particular, since $r\mathbb{B}^d\subset \mathcal{K}$, which means $\|x_1\|_2,\|x_2\|_2\ge r$, we know that the angle $\theta\in [0,\pi]$ between $x_1,x_2$ should also be small. Then by the convexity of $\mathcal{K}$, we can show that a small angle implies a small distance after projection, via a geometric argument.

Denote $a=\|\frac{x_1}{\pi_{\mathcal{K}}(x_1)}\|_2, b= \|\frac{x_2}{\pi_{\mathcal{K}}(x_2)}\|_2$, where $a\ge b$. Without loss of generality, by a coordinate change, we may assume for some $\theta\in[0,\pi]$, 
$$
\frac{x_1}{\pi_{\mathcal{K}}(x_1)}=(a,0,\cdots, 0), \quad \frac{x_2}{\pi_{\mathcal{K}}(x_2)}=(b\cos\theta,b\sin \theta, 0, \cdots, 0).
$$
By the law of cosines, we have that
$$
\left\|\frac{x_1}{\pi_{\mathcal{K}}(x_1)}-\frac{x_2}{\pi_{\mathcal{K}}(x_2)}\right\|_2=\sqrt{a^2+b^2-2ab\cos \theta}.
$$
The above expression is convex in $a$ for a fixed $b$ and vice versa. Therefore, the derivative is monotone. Thus, it's upper bounded by 
\begin{align*}
\max\{\sqrt{b^2+b^2-2b^2\cos \theta},\sqrt{R^2+b^2-2Rb\cos \theta}\}
\end{align*}
for any fixed $b$. Similarly, it's bounded by 
\begin{align*}
\max\{\sqrt{r^2+a^2-2ra\cos \theta},\sqrt{a^2+a^2-2a^2\cos \theta}\}
\end{align*}
for any fixed $a$. We discuss two cases of $\theta$.

\textbf{Case A, $\cos \theta\le \frac{r}{a}$:}
in this case, we have that
$$
x_2^{\top}(\frac{x_1}{\pi_{\mathcal{K}}(x_1)}-\frac{rx_2}{\|x_2\|_2})\le 0,
$$
which is equivalent to $\cos \theta(a-r\cos \theta)-r\sin^2 \theta\le 0$. As a result, for a fixed $a$, the maximum of distance is taken at $b=a$, therefore
$$
\sqrt{r^2+a^2-2ra\cos \theta}\le \sqrt{a^2+a^2-2a^2\cos \theta}\le R\theta,
$$
since $1-\cos \theta=2\sin^2 \frac{\theta}{2}$, and $\sin \theta\le \theta$ for $\theta\ge 0$.

\textbf{Case B, $\cos \theta> \frac{r}{a}$:} 
in this case, the estimate $\sqrt{r^2+a^2-2ra\cos \theta}$ is too rough, and we will make use of the assumption to obtain a lower bound on $b$.
Because $\mathcal{K}$ is convex, the triangle formed by 
$$(a,0,\cdots, 0),(\frac{r^2}{a},\frac{r\sqrt{a^2-r^2}}{a},\cdots, 0),(0,0,\cdots, 0)$$
is a subset of $\mathcal{K}$. Slightly abusing the notations, denote $x,y$ as the first two coordinates for now. We denote $P$ as the intersection point between the two lines
$$
y=\tan \theta \times x, y=-\frac{r}{\sqrt{a^2-r^2}}(x-a).
$$
Then we have that
$$
P=(\frac{ra}{\sqrt{a^2-r^2}\tan \theta+r},\frac{\tan \theta \times ra}{\sqrt{a^2-r^2}\tan \theta+r}),
$$
therefore $b\ge \|P\|_2=\frac{ra}{\sqrt{a^2-r^2}\sin \theta +r\cos \theta}$. As a result, inserting $b$ into the law of cosines, the distance is at most
$$
\max\{\frac{a^2\sin \theta}{\sqrt{a^2-r^2}\sin \theta +r\cos \theta}, \sqrt{2a^2(1-\cos \theta)}\}\le \frac{R^3 \theta}{r^2}.
$$

Now we only need to establish an upper bound on $\theta$. We set $C=\frac{r}{100}$ in Lemma \ref{lem lip}. Since $\|x_1-x_2\|_2\le C$, we have that $\sin \frac{\theta}{2}\le \frac{\|x_1-x_2\|_2}{2r}\le \frac{1}{200}$, and in this regime, 
$$
\theta \le 4 \sin \frac{\theta}{2} \le \frac{2}{r} \|x_1-x_2\|_2.
$$
Combining the two bounds, we have that
$$
\|\frac{x_1}{\pi_{\mathcal{K}}(x_1)}-\frac{x_2}{\pi_{\mathcal{K}}(x_2)}\|_2\le \frac{2}{r}\max\{R, \frac{R^3}{r^2}\} \|x_1-x_2\|_2\le 2\kappa^3 \|x_1-x_2\|_2.
$$

\paragraph{Case 3, $x_1\in \mathcal{K}$, $x_2\notin \mathcal{K}$:} we denote $a=\|x_1\|_2\le R$, $b=\|x_2\|_2\ge r$ and $c=\|\frac{x_2}{\pi(x_2)}\|_2\ge r$. We make an argument based on the angle $\theta\in [0,\pi]$ between $x_1,x_2$. Without loss of generality, by a coordinate change, we may assume 
$$
x_1=\frac{x_1}{\pi_{\mathcal{K}}(x_1)}=(a\cos \theta,a\sin \theta,\cdots, 0),\frac{x_2}{\pi_{\mathcal{K}}(x_2)}=(c, 0, \cdots, 0).
$$
We have that
$$
\|x_1-x_2\|_2^2=a^2+b^2-2ab\cos\theta, \|x_1-\frac{x_2}{\pi_{\mathcal{K}}(x_2)}\|_2^2=a^2+c^2-2ac\cos\theta.
$$
Clearly, when $a\cos\theta\le r$, $x_2^{\top}(x_1-\frac{x_2}{\pi_{\mathcal{K}}(x_2)})\le 0$, and $\|\frac{x_1}{\pi_{\mathcal{K}}(x_1)}-\frac{x_2}{\pi_{\mathcal{K}}(x_2)}\|_2\le \|x_1-x_2\|_2$. Therefore we only need to consider the case when $\cos\theta> \frac{r}{a}$ (and $a\ge r$). In this case, $\|x_1-x_2\|_2\ge \sqrt{a^2\sin^2 \theta}=a\sin \theta$. If $a\sin \theta\ge \frac{r}{2}$, then because $\|x_1-\frac{x_2}{\pi_{\mathcal{K}}(x_2)}\|_2\le 2R$, we have that
$$
\|x_1-\frac{x_2}{\pi_{\mathcal{K}}(x_2)}\|_2\le \frac{4R}{r}\|x_1-x_2\|_2.
$$
If $a\sin \theta< \frac{r}{2}$, because $r\mathbb{B}^d\subset \mathcal{K}$, similar to case 2B, we have that 
$$
c\ge \frac{ar(r\cos\theta+\sqrt{a^2-r^2}\sin\theta)}{r^2-a^2\sin^2\theta},
$$
and
$$
\|x_1-\frac{x_2}{\pi_{\mathcal{K}}(x_2)}\|_2^2\le a^2\sin^2 \theta+a^2\sin^2 \theta[\frac{a^2\sin\theta \cos \theta+r\sqrt{a^2-r^2}}{r^2-a^2\sin^2\theta}]^2.
$$
We notice that
$$
a^2\sin\theta \cos \theta+r\sqrt{a^2-r^2}\le \frac{3R^2}{2}= 2\kappa^2 \frac{3r^2}{4}\le 2\kappa^2(r^2-a^2\sin^2\theta).
$$
As a result, 
$$
\|x_1-\frac{x_2}{\pi_{\mathcal{K}}(x_2)}\|_2\le a\sin \theta (1+2\kappa^2)\le (1+2\kappa^2)\|x_1-x_2\|_2.
$$
Notice that both $4\kappa, 1+2\kappa^2$ are upper bounded by $4\kappa^2$, we conclude our proof.
\end{proof}

\subsection{Proof of Lemma~\ref{cor:lip-control}}

\begin{proof}
We will show in the following lemma that there exists a linear mapping between $\mathbb{R}^{d_x}$ and $\mathcal{K}$ that is bijective when restricted to $\Delta^{d_x}$, and $\K$ satisfies the aspherity condition.

\begin{lemma} [Transformation]
\label{lem:transformation}
The sets $\Delta^{d_x}$, $\mathcal{K}$ satisfy the following properties:
\begin{enumerate}
\item There exists a linear map $\phi:\mathbb{R}^{d_x}\rightarrow\mathcal{K}$ that is bijective when restricted to $\Delta^{d_x}$. 
\item $\phi$ is $1$-non-expanding w.r.t. both $(\|\cdot\|_1,\|\cdot\|_1)$ and $(\|\cdot\|_2, \|\cdot\|_2)$. 
Let $\phi^{-1}:\K\rightarrow\Delta^{d_x}$ be the inverse of $\phi$ restricted to $\Delta^{d_x}$. $\phi^{-1}$ is $2$-non-expanding w.r.t. $(\|\cdot\|_1, \|\cdot\|_1)$ and $d_x$-non-expanding w.r.t. $(\|\cdot\|_2, \|\cdot\|_2)$. 
\item $\K$ is $\kappa$-aspherical with $\kappa=R/r=2d_x$, where $R=1$, and $r=1/(2d_x)$.
\end{enumerate}
\end{lemma}

As a result, (1) follows from that $\pi$ is $(1/r)$-Lipschitz (\cite{lu2023projection}) and $r=2d_x$ (Lemma~\ref{lem:transformation}). (2) follows from Lemma~\ref{lem:extension-lip} with $\|\cdot\|=\|\cdot\|_1$ and $\kappa=2d_x$ in Lemma~\ref{lem:transformation}.
\end{proof}

\subsection{Proof of Lemma~\ref{lem:transformation}}

\begin{proof}
The linear map $\phi$ is simply given by $\forall u=(u_1,\dots,u_{d_x})\in\mathbb{R}^{d_x}$,
\begin{align}
\label{eq:transformation-map}
(u_1,\dots,u_{d_x})\mapsto (u_1,\dots,u_{d_x-1})-\frac{\mathbf{1}_{d_x-1}}{2(d_x-1)}.
\end{align}
It is easy to see that $\phi$ is bijective when restricted to $\Delta^{d_x}$, and $\phi$ is $1$-non-expanding; $\phi^{-1}$ is $2$-non-expanding w.r.t. $(\|\cdot\|_1, \|\cdot\|_1)$ and $d_x$-non-expanding w.r.t. $(\|\cdot\|_2, \|\cdot\|_2)$.

We move on to show the aspherity of $\K$. Fix $v\in\mathbb{R}^{d_y-1}$ such that $\|v\|_2\le r$. By choice of $r$, we have
\begin{align*}
&\min_{1\le i\le d_x-1}e_i^{\top}\left(\frac{\mathbf{1}_{d_x-1}}{2(d_x-1)}+v\right)\ge -r+\frac{1}{2(d_x-1)}\ge  0,\\
&\left\|v+\frac{\mathbf{1}_{d_x-1}}{2(d_x-1)}\right\|_1=\|v\|_1+\frac{1}{2}\le \sqrt{d_x-1}\|v\|_2+\frac{1}{2}\le 1.
\end{align*}
To see the upper bound, we have that $\forall u\in\K$, $u=v+\frac{\mathbf{1}_{d_x-1}}{2(d_x-1)}$, $v\ge 0, \|v\|_1\le 1$ and thus satisfies
\begin{align*}
\|u\|_2=\sqrt{\sum_{i=1}^{d_x-1}\left(v_i-\frac{1}{2(d_x-1)}\right)^2}\le 1-\frac{1}{2(d_x-1)}\le R. 
\end{align*}
\end{proof}

\newpage

\section{Proof of Theorem~\ref{thm:po-main}}

\subsection{Proof of Observation~\ref{obs:ext}}
\begin{proof}
$e_t$ is convex since $c_t^y$ is convex, $c_t^u\circ\phi^{-1}$ is a convex function (composition of convex and linear functions) over $\mathcal{K}$, and thus the Minkowski extension of $c_t^u\circ\phi^{-1}$ on $\K$ is convex (Lemma~\ref{lem:extension-property}). The inequality follows from the third property in Lemma~\ref{lem:extension-property}.
\begin{align*}
e_t(M)&\ge c_t^y(y_t(M; K_0))+c_t^u(\phi^{-1}(\frac{\phi(\tilde{u}_t(M;K_0))}{\pi_{\mathcal{K}}(\phi(\tilde{u}_t(M;K_0))}))\\
&=c_t^y(y_t(M; K_0))+c_t^u(u_t(M; K_0))\\
&=c_t(y_t(M;K_0),u_t(M;K_0)).
\end{align*}
\end{proof}

\subsection{Signal Realizability}
Unlike $\GPCS$ (and $\GPC$), the signals that $\GPCPOS$ combines to construct the control at each step are not the disturbances $w_{t-1:t-H}$, which are no longer identifiable due to the observation mechanism. Instead, like the partially observable variant of $\GPC$ \cite[Chapter 9]{hazan2022introduction}, the signals are the counterfactual observations that would have been observed had a ``default'' linear policy $y \mapsto K_0 y$ been followed for all time. Indeed, the following lemma shows that the signal $o_t$ computed in Line~\ref{line:compute-signal} of \cref{alg:gpc-po} is exactly $y_t(K_0)$ where $K_0$ is chosen (arbitrarily) in Line~\ref{line:choose-k0}.

\begin{lemma} [Signal realizability]
\label{lemma:signal-realizability}
Let $d_x,d_y,T \in \BN$ and let $\ML = (A,B,C,x_1,(\gamma_t)_t,(w_t)_t,(c_t)_t,\alpha)$ be a partially observable simplex LDS. For any $u_1,\dots,u_T \in \Delta^{d_x}$, let $y_t(u_{1:t-1})$ denote the observation at time $t$ in $\ML$ if controls $u_{1:t-1}$ were played at times $1,\dots,t-1$. Then for each $t$ it holds that
\begin{equation} y_t(u_{1:t-1}) = \sum_{i=1}^{t-1} \bar\lambda_{t,i}C((1-\alpha)A)^{i-1}\alpha Bu_{t-i} + \sum_{i=1}^t \lambda_{t,i}C((1-\alpha)A)^{i-1} w_{t-i},\label{eq:yt-unfold-po}\end{equation}
where $w_0:=x_1$.
Thus, for any $K_0 \in \BS^{d_y,d_x}$ and $t \in [T]$, it holds that
\begin{align}
  \label{eq:yt-unfold-po-2}y_t(u_{1:t-1}) - \sum_{i=1}^{t-1} \bar\lambda_{t,i}C((1-\alpha)A)^{i-1}\alpha B(u_{t-i} - K_0 y_{t-i}(K_0)) = y_t(K_0).\end{align}
\end{lemma}

\begin{proof}
It suffices to show that $\forall t$,
\begin{equation} x_t(u_{1:t-1}) = \sum_{i=1}^{t-1} \bar\lambda_{t,i}((1-\alpha)A)^{i-1}\alpha Bu_{t-i} + \sum_{i=1}^t \lambda_{t,i}((1-\alpha)A)^{i-1} w_{t-i}.\label{eq:xt-unfold-po}\end{equation}
We prove \cref{eq:xt-unfold-po} via induction on $t\in\mathbb{N}$. When $t=1$, $x_1=\lambda_{1,1}w_0$ by definition. Suppose for $t\in\mathbb{N}$, \cref{eq:xt-unfold-po} holds. Then,
\begin{align*}
x_{t+1}(u_{1:t})&=(1-\gamma_t)(1-\alpha)Ax_{t}(u_{1:t-1})+(1-\gamma_ t)\alpha Bu_t+\gamma_tw_t\\
&=(1-\gamma_t)\sum_{i=2}^{t}\bar{\lambda}_{t,i-1}((1-\alpha)A)^{i-1}\alpha Bu_{t+1-i}+(1-\gamma_t)\alpha Bu_t\\
& \quad +(1-\gamma_t)\sum_{i=2}^{t+1}\lambda_{t,i-1}((1-\alpha)A)^{i-1}w_{t+1-i}+\gamma_tw_t\\
&=\sum_{i=1}^{t}\bar{\lambda}_{t+1,i}((1-\alpha)A)^{i-1}\alpha Bu_{t+1-i}+\sum_{i=1}^{t+1}\lambda_{t+1,i}((1-\alpha)A)^{i-1}w_{t+1-i},
\end{align*}
where the last equality uses that $(1-\gamma_t)\bar{\lambda}_{t,i-1}=\bar{\lambda}_{t+1,i}$, $(1-\gamma_t)=\bar{\lambda}_{t+1,1}$, $(1-\gamma_t)\lambda_{t,i-1}=\lambda_{t+1,i}$, and $\gamma_t=\lambda_{t+1,1}$. 
\end{proof}

\subsection{Approximation of Linear Markov Policies}
\label{sec:approx-po}
We start by proving that every policy in the comparator class can be approximated by $\pi^M$ for some $M \in \MM_{d_x,d_y,H}$.  
We will also use the following equality:

%Similar to the fully observable setting, we need to show that the learning class $\mathcal{M}$ defined in \cref{def:learning-class} approximates the comparator class defined in \cref{def:po-comparator} well. 

\begin{observation}
\label{fact:linear-unfold}
For any $K \in \BS^{d_x,d_y}$ and $t \in [T]$, the observation in $\ML$ at time $t$ under policy $y \mapsto Ky$ can be written as
\begin{align}
\label{eq:K0-signal-unfold}
y_t(K)&=\sum_{i=1}^{t}\lambda_{t,i}C(\mathbb{A}_{K}^{\alpha})^{i-1}w_{t-i}.
\end{align}
Moreover, we have $u_t(K) = K \cdot y_t(K)$. 
\end{observation}

\begin{lemma} [Approximation]
\label{lem:po-approx}
Suppose that the cost functions $c_1,\dots,c_T$ of $\ML$ satisfy \cref{asm:po-convex-loss} with Lipschitz parameter $L>0$. Fix $\tau,\eps>0$ and suppose $H$ satisfies $H\ge 2\tau\lceil \log(512d_x^{7/2}LT^2/\eps)\rceil$. For any $K^{\star}\in\Ksim_\tau$ and $K_0\in\BS^{d_x,d_y}$, there is some $M\in\MM_{d_x,d_y,H}$ such that
\begin{align*}
\left|\sum_{t=1}^T e_t(M)-\sum_{t=1}^Tc_t(y_t(K^{\star}),u_t(K^{\star}))\right|\le \eps.
\end{align*}
\end{lemma}

\begin{proof}
We break down the proof into two parts. We will construct $M\in\MM_{d_x,d_y,H}$ and show separately that the following two inequalities hold.
\begin{align}
\left|\sum_{t=1}^T c_t(y_t(M;K_0),u_t(M;K_0))-\sum_{t=1}^T c_t(y_t(K^{\star}), u_t(K^{\star}))\right|\le \frac{\eps}{2} \label{eq:cost-approx},\\
\left|\sum_{t=1}^T e_t(M)-\sum_{t=1}^T c_t(y_t(M;K_0),u_t(M;K_0))\right|\le \frac{\eps}{2} \label{eq:ext-approx}.
\end{align}
\paragraph{Proof of \cref{eq:cost-approx}.} We start with \cref{eq:cost-approx}. For each $M$, recall that $\tilde{u}_t(M;K_0)$ is the control parameterized by $M$ using signals generated by the linear policy $K_0$ before the Minkowski projection. Formally, this is defined as
\begin{align*}
\tilde{u}_t(M;K_0) := M^{[0]}y_t(K_0)+\alpha\sum_{i=1}^H \bar{\lambda}_{t,i}M^{[i]}y_{t-i}(K_0),
\end{align*}
so that $u_t(M;K_0)$ is obtained by the Minkowski projection of $(y_t(M;K_0),\tilde{u}_t(M;K_0))$. First, we show that there exists $M$ such that the sequence $(\tilde{u}_t(M;K_0))_t$ approximates $(u_t(K^\st))_t$ well. Second, we show that the error induced by the projection operation is small. Third, we use Lipschitzness of the cost functions to conclude. 

Concretely, define $M = M^{[0:H]}$ by $M^{[0]} := K^\st$ and $M^{[i]} := K^{\star}C(\mathbb{A}_{K^{\star}}^{\alpha})^{i-1}B(K^{\star}-K_0)$ for each $i \in [H]$. Note that since $K^{\star}C(\mathbb{A}_{K^{\star}}^{\alpha})^{i-1}BK^{\star}\in\mathbb{S}^{d_x, d_y}$ and $ K^{\star}C(\mathbb{A}_{K^{\star}}^{\alpha})^{i-1}BK_0 \in \BS^{d_x,d_y}$, we have that $M^{[i]}\in\mathbb{S}_{0}^{d_x,d_y}$ for each $i \in [H]$, and hence $M\in \MM_{d_x,d_y,H}$. We will show that the following bound hold:
\begin{align}
\|\tilde{u}_t(M;K_0)-u_t(K^{\star})\|_1&\le \frac{\eps}{512d_x^{7/2}LT^2}, \label{eq:utilde-uk}
\end{align}
and use \cref{eq:utilde-uk} to show that
\begin{align}
\|u_t(M;K_0)-u_t(K^{\star})\|_1&\le \frac{\eps}{8LT^2}\label{eq:u-uk}.
\end{align}
We start with \cref{eq:utilde-uk}. For notational convenience, define $M^{[i]} := K^{\star}C(\mathbb{A}_{K^{\star}}^{\alpha})^{i-1}B(K^{\star}-K_0)$ for each $H<i\le t-1$, and define
\begin{align*}
\hat{u}_{t}(M;K_0):=M^{[0]}y_t(K_0)+\alpha\sum_{i=1}^{t-1}\bar{\lambda}_{t,i}M^{[i]}y_{t-i}(K_0).
\end{align*}
Then by applying Observation~\ref{fact:linear-unfold} to $K_0$, we have
\begin{align*}
& \quad \hat{u}_{t}(M;K_0) \\
&=M^{[0]}y_t(K_0)+\alpha\sum_{i=1}^{t-1}\sum_{j=i+1}^{t}\bar{\lambda}_{t,i}\lambda_{t-i,j-i}M^{[i]}C(\mathbb{A}_{K_{0}}^{\alpha})^{j-i-1}w_{t-j}\\
&=M^{[0]}y_t(K_0)+\alpha\sum_{j=1}^{t}\left(\sum_{i=1}^{j-1}\bar{\lambda}_{t,i}\lambda_{t-i,j-i}M^{[i]}C(\mathbb{A}_{K_{0}}^{\alpha})^{j-i-1}\right)w_{t-j}\\
&\overset{(a)}{=} M^{[0]}y_t(K_0)+\sum_{j=1}^{t}\lambda_{t,j}K^{\star}C\left(\sum_{i=1}^{j-1}(\mathbb{A}_{K^\st}^{\alpha})^{i-1}\alpha B(K^{\star}-K_0)C(\mathbb{A}_{K_{0}}^{\alpha})^{j-i-1}\right)w_{t-j}\\
&=M^{[0]}y_t(K_0)+\sum_{j=1}^{t}\lambda_{t,j}K^{\star}C\left(\sum_{i=1}^{j-1}(\mathbb{A}_{K_{0}}^{\alpha}+\alpha B(K^{\star}-K_0)C)^{i-1}\alpha B(K^{\star}-K_0)C(\mathbb{A}_{K_{0}}^{\alpha})^{j-i-1}\right)w_{t-j}\\
%& \ \ \ \ \ + \sum_{j=1}^{t-1}\bar{\lambda}_{t,j}K^{\star}C\left[\sum_{i=1}^{j-1}(\tilde{A}+B(K^{\star}-K_0)C)^{i-1}B(K^{\star}-K_0)C\tilde{A}^{j-i-1}\right]BK_0e_{t-j}\\
%& \ \ \ \ \ + \sum_{i=1}^{t-1}\bar{\lambda}_{t,i}K^{\star}C(A+BK^{\star}C)^{i-1}B(K^{\star}-K_0)e_{t-i}\\
&\overset{(b)}{=} M^{[0]}y_t(K_0)+\sum_{j=1}^{t}\lambda_{t,j}K^{\star}C(\mathbb{A}_{K^{\star}}^{\alpha})^{j-1}w_{t-j}-\sum_{j=1}^{t-1}\lambda_{t,j}K^{\star}C(\mathbb{A}_{K_{0}}^{\alpha})^{j-1}w_{t-j}\\
%& \ \ \ \ \ + \sum_{j=1}^{t-1}\bar{\lambda}_{t,j}K^{\star}C(A+BK^{\star}C)^{j-1}BK^{\star}e_{t-j}-\sum_{j=1}^{t-1}\bar{\lambda}_{t,j}K^{\star}C(A+BK_0C)^{j-1}BK_0e_{t-j}\\
&\overset{(c)}{=} \sum_{j=1}^{t}\lambda_{t,j}K^{\star}C(\mathbb{A}_{K^{\star}}^{\alpha})^{j-1}w_{t-j}\\%+\sum_{j=1}^{t-1}\bar{\lambda}_{t,j}K^{\star}C(A+BK^{\star}C)^{j-1}BK^{\star}e_{t-j}+K^{\star}e_t\\
&\overset{(d)}{=}u_t(K^{\star}),
\end{align*}
where equality $(a)$ uses the definition of $M^{[i]}$ together with the fact that
\begin{align}
\label{eq:lambda-property}
\bar{\lambda}_{t,i}\lambda_{t-i, j-i}&=\left(\prod_{k=1}^i (1-\gamma_{t-k})\right)\gamma_{t-j}\cdot \prod_{k=1}^{j-i-1}(1-\gamma_{t-i-k}) \nonumber\\
&=\gamma_{t-j}\left(\prod_{k=1}^i (1-\gamma_{t-k})\right)\left(\prod_{k=i+1}^{j-1} (1-\gamma_{t-k})\right) \nonumber\\
&=\gamma_{t-j}\prod_{k=1}^{j-1}(1-\gamma_{t-k}) \nonumber\\
&=\lambda_{t,j}
\end{align}
by \cref{eq:define-lambdas}; equality $(b)$ follows from the matrix equality
\begin{align*}
X^n=Y^n+\sum_{i=1}^n X^{i-1}(X-Y)Y^{n-i},
\end{align*}
with 
\[X=(\mathbb{A}_{K_{0}}^{\alpha}+\alpha B(K^{\star}-K_0)C)=\mathbb{A}_{K^{\star}}^{\alpha}, \ \ \ Y=\mathbb{A}_{K_{0}}^{\alpha},\]
equality $(c)$ uses Observation~\ref{fact:linear-unfold} with $K_0$ together with the choice of $M\^0 = K^\star$; and equality $(d)$ uses Observation~\ref{fact:linear-unfold} with $K^\st$.

Let $p^{\star}\in\Delta^{d_x}$ be the unique stationary distribution of $\mathbb{A}_{K^{\star}}^{\alpha}$. Since $u_t(K^{\star})=\hat{u}_{t}(M;K_0)$, we have
\begin{align*}
\|u_{t}(K^{\star})-\tilde{u}_{t}(M;K_0)\|_1&\leq\left\|\sum_{i=H+1}^{t-1}\bar{\lambda}_{t,i}M^{[i]}y_{t-i}(K_0)\right\|_1\\
&=\left\|\sum_{i=H+1}^{t-1}\bar{\lambda}_{t,i}K^{\star}C(\mathbb{A}_{K^{\star}}^{\alpha})^{i-1}B(K^{\star}-K_0)y_{t-i}(K_0)\right\|_1\\
&\le \left\|\sum_{i=H+1}^{t-1}\left(\bar{\lambda}_{t,i}K^{\star}C(\mathbb{A}_{K^{\star}}^{\alpha})^{i-1}\left(BK^{\star}y_{t-i}(K_0)\right)-\bar{\lambda}_{t,i}K^{\star}Cp^{\star}\right)\right\|_1\\
&\qquad+ \left\|\sum_{i=H+1}^{t-1}\left(\bar{\lambda}_{t,i}K^{\star}C(\mathbb{A}_{K^{\star}}^{\alpha})^{i-1}\left(BK_0y_{t-i}(K_0)\right)-\bar{\lambda}_{t,i}K^{\star}Cp^{\star}\right)\right\|_1\\
&\leq \sum_{i=H+1}^{t-1} \bar\lambda_{t,i} \norm{K^\st C (\mathbb{A}_{K^\st}^\alpha)^{i-1}(BK^\st y_{t-i}(K_0) - p^\st)}_1 \\ 
&\qquad+ \sum_{i=H+1}^{t-1} \bar\lambda_{t,i} \norm{K^\st C (\mathbb{A}_{K^\st}^\alpha)^{i-1}(BK_0 y_{t-i}(K_0) - p^\st)}_1 \\
&\leq \sum_{i=H+1}^{t-1} \bar\lambda_{t,i} \norm{(\mathbb{A}_{K^\st}^\alpha)^{i-1}(BK^\st y_{t-i}(K_0) - p^\st)}_1 \\ 
&\qquad+ \sum_{i=H+1}^{t-1} \bar\lambda_{t,i} \norm{ (\mathbb{A}_{K^\st}^\alpha)^{i-1}(BK_0 y_{t-i}(K_0) - p^\st)}_1 \\
&\le 2\sum_{i=H+1}^{t-1} 2^{-H/\tau} \\
&\leq \frac{\eps}{512d_x^{7/2}LT^2}.
\end{align*}
where the penultimate inequality is by Lemma 18 in \citep{golowich2024online} and the final inequality is by choice of $H$. This proves \cref{eq:utilde-uk}. 

Recall that $u_t(M;K_0)$ is obtained via the operation
\begin{align*}
u_t(M;K_0)=\phi^{-1}(\frac{\phi( \tilde{u}_t(M;K_0))}{\pi_{\mathcal{K}}(\phi( \tilde{u}_t(M;K_0))}).
\end{align*}
We have that
\begin{align*}
\|u_t(M;K_0)-u_t(K^{\star})\|_1&=\left\|\phi^{-1}(\frac{\phi( \tilde{u}_t(M;K_0))}{\pi_{\mathcal{K}}(\phi( \tilde{u}_t(M;K_0))})-\phi^{-1}(\frac{\phi( u_t(K^{\star}))}{\pi_{\mathcal{K}}(\phi( u_t(K^{\star}))})\right\|_1\\
&\le 2\cdot \left\|\frac{\phi( \tilde{u}_t(M;K_0))}{\pi_{\mathcal{K}}(\phi( \tilde{u}_t(M;K_0))}-\frac{\phi( u_t(K^{\star}))}{\pi_{\mathcal{K}}(\phi( u_t(K^{\star}))}\right\|_1\\
&\le 64 d_x^{7/2}\cdot \|\tilde{u}_t(M;K_0))-u_t(K^{\star})\|_1\\
&\le \frac{\eps }{8LT^2},
\end{align*}
where the first inequality follows from that $\phi^{-1}$ is $2$-non-expanding with respect to $(\|\cdot\|_1, \|\cdot\|_1)$ (Lemma~\ref{lem:transformation}); the second inequality follows from Lemma~\ref{thm:proj-lip} by taking $\kappa=2d_x$; and that $\phi$ is $1$-non-expanding with respect to $(\|\cdot\|_1, \|\cdot\|_1)$ (Lemma~\ref{lem:transformation}); the third inequality follows from \cref{eq:utilde-uk}.

Thus, we proved \cref{eq:u-uk}, which
holds for all $t$. We next bound $\|y_t(K^{\star})-y_t(M;K_0)\|_1$.
From \cref{eq:yt-unfold-po-2} and \cref{eq:u-uk}, we have
\begin{align*}
\|y_t(K^{\star})-y_t(M;K_0)\|_1
&\le \sum_{i=1}^{t-1}\bar{\lambda}_{t,i}\left\|C[(1-\alpha)A]^{i-1}B\left[u_{t-i}(K^{\star})-u_{t-i}(M;K_0)\right]\right\|_1 \\ 
&\leq \frac{\vep}{8 LT^2} \sum_{i=1}^{t-1}\bar\lambda_{t,i} \leq \frac{\vep}{8LT}.
\end{align*}
Combining the bounds on $\|u_{t}(K^{\star})-u_t(M;K_0)\|_1$ and $\|y_t(K^{\star})-y_t(M;K_0)\|_1$ and using the fact that $c_t$ satisfies \cref{asm:po-convex-loss} for all $t$, we have that the approximation error is bounded by
\begin{align}
\label{eq:c-approx}
\left|\sum_{t=1}^T c_t(y_t(M;K_0),u_t(M;K_0))-\sum_{t=1}^Tc_t(y_t(K^{\star}),u_t(K^{\star}))\right|\le \frac{\eps}{4}.
\end{align}
\paragraph{Proof of \cref{eq:ext-approx}.} We move on to bound the difference between the pseudo loss and the cost function in \cref{eq:ext-approx}. 
\begin{align*}
& \quad \left|\sum_{t=1}^T e_t(M)-\sum_{t=1}^T c_t(y_t(M;K_0),u_t(M;K_0))\right|\\
&\le \underbrace{\left|\sum_{t=1}^T (\pi_{\mathcal{K}}(\phi(\tilde{u}_t(M;K_0)))-1) \cdot c_t^u(\phi^{-1}(\frac{\phi(\tilde{u}_t(M;K_0))}{\pi_{\mathcal{K}}(\phi(\tilde{u}_t(M;K_0)))}))\right|}_{(1)} \\
& \quad +\underbrace{\left|\sum_{t=1}^T c_t^u(\phi^{-1}(\frac{\phi(\tilde{u}_t(M;K_0))}{\pi_{\mathcal{K}}(\phi(\tilde{u}_t(M;K_0)))}))- \sum_{t=1}^T c_t^u(u_t(M;K_0))\right|}_{(2)}.
\end{align*}
By Lipschitz property of $\pi_{\mathcal{K}}$ (Lemma~\ref{cor:lip-control}) and non-expanding condition of $\phi$ (Lemma~\ref{lem:transformation}) and that $\phi(u_t(K^{\star}))\in\K$, we have that
\begin{align*}
|\pi_{\mathcal{K}}(\phi( \tilde{u}_t(M;K_0)))-1|&=
|\pi_{\mathcal{K}}(\phi(\tilde{u}_t(M;K_0)))-\pi_{\mathcal{K}}(\phi( u_t(K^{\star})))|\\
&\le 2d_x\cdot \|\tilde{u}_t(M;K_0)-u_t(K^{\star})\|_1\\
&\le \frac{\eps }{256d_x^{5/2}LT^2},
\end{align*}
where the last inequality follows from \cref{eq:utilde-uk}. 
By the Lipschitz condition of $c_t^u$ and that $c_t^u$ attains $0$ on $\mathcal{V}$ (\cref{asm:po-convex-loss}), we have
\begin{align*}
(1)\le L \sum_{t=1}^T (\pi_{\mathcal{K}}(\phi(\tilde{u}_t(M;K_0)))-1) \le \frac{\eps}{256d_x^{5/2}T}.
\end{align*}
We move on to bound $(2)$:
\begin{align*}
(2)&\le L\sum_{t=1}^T \left\|\phi^{-1}(\frac{\phi(\tilde{u}_t(M;K_0))}{\pi_{\mathcal{K}}(\phi(\tilde{u}_t(M;K_0)))})-u_t(M;K_0)\right\|_1\\
&\le 2L\sum_{t=1}^T \left\|\frac{\phi(\tilde{u}_t(M;K_0))}{\pi_{\mathcal{K}}(\phi(\tilde{u}_t(M;K_0)))}-\phi(u_t(M;K_0))\right\|_1\\
&\le 2L\sum_{t=1}^{T}  \frac{\|\phi(\tilde{u}_t(M;K_0))-\phi(u_t(M;K_0))\|_1}{\pi_{\mathcal{K}}(\phi(\tilde{u}_t(M;K_0)))} \\
&\quad + 2L\sum_{t=1}^{T}  \left(1-\frac{1}{\pi_{\mathcal{K}}(\phi(\tilde{u}_t(M;K_0)))}\right)\cdot \|\phi(u_t(M;K_0))\|_1\\
&\le 2L\sum_{t=1}^T \|\tilde{u}_t(M;K_0)-u_t(M;K_0)\|_1+\frac{\eps}{64d_x^{5/2}T}\\
&\le \frac{\eps}{2T}+\frac{\eps}{64d_x^{5/2}T},
\end{align*}
where the first inequality follows from the Lipschitz assumption on $c_t^u$; the second inequality follows from that $\phi^{-1}$ is $2$-non-expanding w.r.t. ($\|\cdot\|_1$, $\|\cdot\|_1$) on $\K$; the second to last inequality follows from $\pi_{\mathcal{K}}\ge 1$, $\phi$ is $1$-non-expanding w.r.t. ($\|\cdot\|_1,\|\cdot\|_1$), and that $\pi_{\mathcal{K}}(\phi(\tilde{u}_t(M;K_0)))\le 2$, $\|\phi(u_t(M;K_0))\|_1=1$; the last inequality follows from \cref{eq:utilde-uk} and \cref{eq:u-uk}.

Combining, we have that
\begin{align}
\label{eq:e-approx}
\left|\sum_{t=1}^T e_t(M)-\sum_{t=1}^T c_t(y_t(M;K_0),u_t(M;K_0))\right|\le \frac{3\eps}{4}.
\end{align} 
Combining \cref{eq:c-approx} and \cref{eq:e-approx} gives the desired bound.
\end{proof}

\subsection{Bounding the Memory Mismatch Error}
\label{sec:mem-mismatch-po}
In this section we prove the analogue of Lemma 21 in \citep{golowich2024online} for the partially observable setting. %Analogous to \cref{sec:memory-mismatch}, in this section, we prove \cref{lem:po-mem-mismatch}, which allows us to show that an algorithm with bounded aggregate loss with respect to the loss functions $\ell_t$ defined on \cref{line:po-lt} of \cref{alg:gpc-po} in fact has bounded aggregate cost with respect to the cost functions $c_t$ chosen by the adversary. 
\begin{lemma} [Memory mismatch error]
\label{lem:po-mem-mismatch}
Suppose that $(c_t)_t$ satisfy \cref{asm:po-convex-loss} with Lipschitz parameter $L$. Let $\tau,\beta > 0$, and suppose that $\Ksim_\tau(\ML)$ is nonempty. Consider the execution of $\GPCPOS$ (\cref{alg:gpc-po}) on $\ML$. If the iterates $(M_t\^{0:H})_{t \in [T]}$ satisfy
 \begin{align}
   % \gpcnorm{(p_t, M_t\^{1:H}) - (p_{t+1}, M_{t+1}\^{1:H})} \leq \beta
   \max_{1 \leq t \leq T-1} \max_{i \in [H]} \oneonenorm{M_t\^i - M_{t+1}\^i} \leq \beta,
   \label{eq:ptmt-change-po}
 \end{align}
then for each $t \in [T]$, the loss function $\ell_t$ computed at time step $t$ satisfies
    \begin{align}
| c_t(y_t(M_t;K_0),u_t(M_t;K_0)) - c_t(y_t, u_t)| &\leq O\left(L d_x^{7/2} \beta\log^2(1/\beta)H\tau^2\right)\nonumber.
    \end{align}
\iffalse
Assume that \cref{alg:gpc-po} satisfies
\begin{align*}
\max_{1\le t\le T-1}\max_{0\le i\le H} \|M_{t}^{[i]}-M_{t+1}^{[i]}\|_{1\rightarrow1}\le \beta,
\end{align*}
then we have
\begin{align*}
\left|\sum_{t=1}^T \ell_t(M_t)-\sum_{t=1}^T c_t(y_t,u_t)\right|\le O(L\tau^2\beta\log^2(1/\beta)HT).
\end{align*}
\fi
\end{lemma}

\begin{proof} 
Fix $t \in [T]$. First, by \cref{eq:yt-unfold-po}, we have that for all $M\in\MM_{d_x,d_y,H}$, %\noah{instances of $y_t(M)$ should be $y_t(M; K_0)$ (and same for $u_t(M)$)?}
\begin{align}
\label{eq:yt-diff}
y_t-y_t(M;K_0) &= \sum_{i=1}^{t-1}\bar{\lambda}_{t,i}C ((1-\alpha)A)^{i-1}\alpha B(u_{t-i}-u_{t-i}(M;K_0)). 
\end{align}
We consider two cases of the mixing time $\tau_A$ of $A$. Let $\tau_A=t^{\mathrm{mix}}(A)$, and $t_0=\lceil \tau_A\log(2/\beta)\rceil$. 

\paragraph{Case 1: $\tau_A\le 4\tau$.} Let $p_{A}^{\star}$ satisfy $p_{A}^{\star}=Ap_{A}^{\star}$ be a stationary distribution of $A$, then by adding and subtracting $p_{A}^{\star}$, we have $\forall M\in\MM_{d_x,d_y,H}$,
\begin{align*}
&\left\|\sum_{i=t_0+1}^{t-1}\bar{\lambda}_{t,i}[(1-\alpha)A]^{i-1}\alpha B(u_{t-i}-u_{t-i}(M;K_0))\right\|_1\\
&\le \left\|\alpha \sum_{i=t_0+1}^{t-1}\bar{\lambda}_{t,i}[(1-\alpha)A]^{i-1}(Bu_{t-i}-p_{A}^{\star})\right\|_1+ \left\|\alpha\sum_{i=t_0+1}^{t-1}\bar{\lambda}_{t,i}[(1-\alpha)A]^{i-1}(Bu_{t-i}(M;K_0)-p_{A}^{\star})\right\|_1\\
&\le \alpha \sum_{i=t_0+1}^{t-1}\bar{\lambda}_{t,i}(1-\alpha)^{i-1}\|A^{i-1}(Bu_{t-i}-p_{A}^{\star})\|_1+ \alpha\sum_{i=t_0+1}^{t-1}\bar{\lambda}_{t,i}(1-\alpha)^{i-1}\|A^{i-1}(Bu_{t-i}(M;K_0)-p_{A}^{\star})\|_1\\
&\le 2 \sum_{i=t_0+1}^{t-1} \frac{1}{2^{\lfloor (i-1)/\tau_A\rfloor}}\le 2\tau_A\beta\sum_{i=0}^{\infty}\frac{1}{2^i}\le C\tau\beta\log(1/\beta).
\end{align*}
Thus, it suffices to show a bound for the first $t_0$ terms. We have that
\begin{align*}
\|\tilde{u}_{t-i}-\tilde{u}_{t-i}(M_t;K_0)\|_1&\le \|(M_{t-i}^{[0]}-M_{t}^{[0]})y_{t-i}(K_0)\|_1+\sum_{j=1}^{H}\bar{\lambda}_{t-i,j}\|(M_{t-i}^{[j]}-M_{t}^{[j]})y_{t-i-j}(K_0)\|_1\\
&\le (H+1)i\beta\le 2Hi\beta,
\end{align*}
since $y_t\in\Delta^{d_y}$, $\forall t$, and $\max_{0\le j\le H}\|M_{t-i}^{[j]}-M_{t}^{[j]}\|_{1\rightarrow 1}\le i\beta$ by the assumption. To bound the $\|u_{t-i}-u_{t-i}(M_t;K_0)\|_1$, we note that by the non-expanding condition established in Lemma~\ref{thm:proj-lip}, we have that
\begin{align*}
\|u_{t-i}-u_{t-i}(M_t;K_0)\|_1&\le 2\sqrt{d_x}\left\|u_{t-i}-u_{t-i}(M_t;K_0)\right\|_2\\
&\le C d_x^{7/2}\|\tilde{u}_{t-i}-\tilde{u}_{t-i}(M_t;K_0)\|_1\\
&\le Cd_x^{7/2}Hi\beta.
\end{align*}
Thus, we can bound $\|y_t-y_t(M_t;K_0)\|_1$ as the following: 
\begin{align*}
\|y_t-y_t(M_t;K_0)\|_1&\le \left\|\sum_{i=1}^{t_0}\bar{\lambda}_{t,i}C[(1-\alpha)A]^{i-1}\alpha B(u_{t-i}-u_{t-i}(M))\right\|_1+C\tau\beta\log(1/\beta)\\
&\le Cd_x^{7/2}\beta H t_0^2 +C\tau\beta\log(1/\beta)\\
&\le Cd_x^{7/2}\beta \log^2(1/\beta) H\tau^2
\end{align*}
for some constant $C$.

\paragraph{Case 2: $\tau_A>4\tau$.} In this case, we will give a lower bound on $\alpha$. Recall that $\tau=t^{\mathrm{mix}}(\mathbb{A}_{K^{\star}}^{\alpha})$, where  $\mathbb{A}_{K^{\star}}^{\alpha}=(1-\alpha)A+\alpha BK^{\star}C$. Suppose (for the purpose of contradiction) that $\alpha\le 1/(128\tau)$. Note that $\|A-\mathbb{A}_{K^{\star}}^{\alpha}\|_{1\rightarrow 1}=\alpha\| A -  BK^{\star}C\|_{1\rightarrow 1}\le \alpha(\|A\|_{1\rightarrow 1}+\|BK^{\star}C\|_{1\rightarrow 1})=2\alpha$. Here, we use Lemma 18 and Lemma 19 from \citep{golowich2024online}, which state the following facts: for $X\in\mathbb{S}^d$ with a unique stationary distribution $\pi$, then (1) $\forall c, t\in\mathbb{N}$, $D_X(t)\le \bar{D}_X(t)\le 2D_X(t)$ and $\bar{D}_X(ct)\le \bar{D}_X(t)^c$, and (2) if $Y\in\mathbb{S}^d$ satisfies $\|X-Y\|_{1\rightarrow 1}\le \delta$, then $\forall t\in\mathbb{N}$, there holds $D_Y(t)\le 2t\delta+2D_X(t)$, where $D_X(t)=\sup_{p\in\Delta^d}\|X^t p-\pi\|_1$ and $\bar{D}_X(t)=\sup_{p,q\in\Delta^d}\|X^t(p-q)\|_1$. These facts implies that $\bar{D}_{\mathbb{A}_{K^{\star}}^{\alpha}}(\tau)\le 2D_{\mathbb{A}_{K^{\star}}^{\alpha}}(\tau)\le 1/2$, and thus $D_{\mathbb{A}_{K^{\star}}^{\alpha}}(4\tau)\le \bar{D}_{\mathbb{A}_{K^{\star}}^{\alpha}}(4\tau)\le (\bar{D}_{\mathbb{A}_{K^{\star}}^{\alpha}}(\tau))^4\le 1/16$. Since $\|A-\mathbb{A}_{K^{\star}}^{\alpha}\|_{1\rightarrow 1}\le 2\alpha$, $D_A(4\tau)\le 2\cdot 4\tau\cdot 2\alpha+2D_{\mathbb{A}_{K^{\star}}^{\alpha}}(4\tau)\le 16\tau\alpha+1/8\le 1/4$, which means that $t^{\mathrm{mix}}(A)\le 4\tau$. Thus, this means that $\alpha > 1/(128\tau)$.

Thus, we have
\begin{align*}
\|y_t-y_t(M_t;K_0)\|_1&\le \sum_{i=1}^{t-1}(1-\alpha)^{i-1}\|u_{t-i}-u_{t-i}(M_t;K_0)\|_1\\
&\le Cd_x^{7/2}H\beta\sum_{i=1}^{t-1}(1-1/(128\tau))^{i-1}\cdot i\\
&\le Cd_x^{7/2}H\beta \tau^2.\\
\end{align*}
%\dhruv{(a minor error here, need to be a bit less lossy and not throw away the $\beta$ factor above)} 
By Lipschitz condition on $c_t$, we have
\begin{align*}
|c_t(y_t(M_t;K_0),u_t(M_t;K_0))-c_t(y_t,u_t)|&\le L(\|y_t-y_t(M_t;K_0)\|_1)\le C L d_x^{7/2} \beta\log^2(1/\beta)H\tau^2.
\end{align*}
\end{proof}

\subsection{Wrapping up the proof of Theorem~\ref{thm:po-main}}\label{sec:po-proof}

Before proving Theorem~\ref{thm:po-main}, we establish that the loss functions $\ell_t$ used in $\texttt{GPC-PO-Simplex}$ are Lipschitz with respect to $\|\cdot\|_{\Sigma}$ in Lemma~\ref{lem:po-lt-lip}, where  $\|\cdot\|_{\Sigma}$ measures the $\ell_1$-norm of $M=M^{[0:H]}$ as a flattened vector in $\mathbb{R}^{d_xd_y(H+1)}$, formally given by
\begin{align*}
\|M_t\|_{\Sigma}:=\sum_{i=0}^{H}\sum_{j\in[d_x],k\in[d_y]}|(M_t^{[i]})_{jk}|.
\end{align*}

\begin{lemma} [Lipschitzness of $e_t$]
\label{lem:po-lt-lip}
Suppose that a partially observable simplex LDS $\ML$ and $H\ge \tau>0$ are given, and $\Ksim_{\tau}(\ML)$ is nonempty. Fix $K_0\in\mathbb{S}^{d_x,d_y}$. For each $t\in[T]$, the pseudo loss function $e_t$ (as defined on Line~\ref{line:et-loss} of \cref{alg:gpc-po} is $O(Ld_x^{9/2}\tau)$-Lipschitz with respect to the norm $\|\cdot\|_{\Sigma}$ in $\MM_{d_x,d_y,H}$.  %\noah{did we say somewhere we're dropping the subscripts on $\mathcal{M}$?}
\end{lemma}

\begin{proof}
Since $e_t$ (Line~\ref{line:et-loss} of \cref{alg:gpc-po}) is a composite of many functions, we analyze the Lipschitz condition of each of its components. First, we show that the $M\mapsto y_t(M;K_0)$ and $M\mapsto \tilde{u}_t(M;K_0)$ are Lipschitz in $M$ on $\MM_{d_x,d_y,H}$.
In particular, fix $M_1,M_2\in\MM_{d_x,d_y,H}$, we will show that 
\begin{align*}
\|y_t(M_1;K_0)-y_t(M_2;K_0)\|_1&\le O(d_x^{7/2}\tau) \|M_1-M_2\|_{\Sigma}, \\ 
\|\tilde{u}_t(M_1;K_0)-\tilde{u}_t(M_2;K_0)\|_1&\le 2\sqrt{d_x} \|M_1-M_2\|_{\Sigma}.
\end{align*}
Expanding $\tilde{u_t}$, we have
\begin{align*}
\|\tilde{u}_t(M_1;K_0)-\tilde{u}_t(M_2;K_0)\|_1&\le \sqrt{d_x}\|\tilde{u}_t(M_1;K_0)-\tilde{u}_t(M_2;K_0)\|_2\\
&\le \sqrt{d_x}\left\|(M_1^{[0]}-M_2^{[0]})y_t(K_0)+\sum_{i=1}^H\bar{\lambda}_{t,i}(M_1^{[i]}-M_2^{[i]})y_{t-i}(K_0)\right\|_1\\
&\le 2\sqrt{d_x} \max_{0\le i\le H} \|M_1^{[i]}-M_2^{[i]}\|_{1\rightarrow 1} \\
&\le 2\sqrt{d_x}\|M_1-M_2\|_{\Sigma},
\end{align*}
where the second inequality follows from the fact that $\ell_2$ norm is bounded by $\ell_1$ norm. By Lemma~\ref{cor:lip-control}, we have
\begin{align*}
\|u_t(M_1;K_0)-u_t(M_2;K_0)\|_1\le O(d_x^{7/2})\|M_1-M_2\|_{\Sigma}.
\end{align*}
Expanding $y_t$, we have that
\begin{align*}
\|y_t(M_1;K_0)-y_t(M_2;K_0)\|_1&=\left\|\sum_{i=1}^{t-1}\bar{\lambda}_{t,i}C[(1-\alpha)A]^{i-1}\alpha B(u_t(M_1)-u_t(M_2))\right\|_1.
\end{align*}
We consider two cases, depending on the mixing time $\tau_A:=\tmix(A)$ of $A$. 

\paragraph{Case 1:} $\tau_A\le 4\tau$. Let the stationary distribution of $A$ be denoted $p^{\star}\in\Delta^{d_x}$. Then by Lemma 22 in \cite{golowich2024online}, $\forall i\in\mathbb{N}$ and $v\in\mathbb{R}^d$ with $\langle \mathbbm{1}, v\rangle=0$, $\|A^{i}v\|_1\le 2^{-\lfloor i/\tau_A\rfloor}\|v\|_1\le 2^{-\lfloor i/4\tau\rfloor}\|v\|_1$. Then,
\begin{align*}
\|y_t(M_1)-y_t(M_2)\|_1&\le \sum_{i=1}^{t-1} \|CA^{i-1}B(u_t(M_1)-u_t(M_2))\|_1\\
&\le \sum_{i=1}^{t-1} 2^{-\lfloor (i-1)/4\tau\rfloor}\|u_t(M_1)-u_t(M_2)\|_1\\
&\le Cd_x^{7/2}\|M_1-M_2\|_{\Sigma}\sum_{i=1}^{t-1} 2^{1-\lfloor (i-1)/4\tau\rfloor}\\
&\le Cd_x^{7/2}\tau \|M_1-M_2\|_{\Sigma}. 
\end{align*}

\paragraph{Case 2:} $\tau_A>4\tau$. As shown in the proof of Lemma~\ref{lem:po-mem-mismatch}, in this case we have $\alpha > 1/(128\tau)$, and therefore
\begin{align*}
\|y_t(M_1)-y_t(M_2)\|_1&\le\sum_{i=1}^{t-1}(1-\alpha)^{i-1}\|(u_t(M_1)-u_t(M_2))\|_1\\
&\le Cd_x^{7/2}\|M_1-M_2\|_{\Sigma}\sum_{i=1}^{t-1}(1-1/(128\tau))^{i-1}\\
&\le C d_x^{7/2}\tau\|M_1-M_2\|_{\Sigma}.
\end{align*}
Here we proved that $M\mapsto y_t(M;K_0)$ is $O(d_x^{7/2}\tau)$-Lipschitz, and $M\mapsto \tilde{u}_t(M;K_0)$ is $2\sqrt{d_x}$-Lipschitz. Immediately, $c_t^y(y_t(M;K_0))$ is $O(Ld_x^{7/2}\tau)$-Lipschitz on $\MM_{d_x,d_y,H}$. It suffices to show that
\begin{align*}
\pi_{\mathcal{K}}(\phi(\tilde{u}_t(M;K_0))) \cdot c_t^u(\phi^{-1}(\frac{\phi(\tilde{u}_t(M;K_0))}{\pi_{\mathcal{K}}(\phi(\tilde{u}_t(M;K_0))}))
\end{align*}
is Lipschitz on $\MM_{d_x,d_y,H}$. Since $\phi$ is $1$-Lipschitz over $\mathbb{R}^{d_y+d_x}$, we have that $M\mapsto \phi(\tilde{u}_t(M;K_0))$ is $2\sqrt{d_x}$-Lipschitz. By assumption $c_t^{u}$ is $L$-Lipschitz on $\Delta^{d_x}$, and $\phi^{-1}$ is $2$-Lipschitz over $\mathcal{K}$, we have $c_t^{u}\circ\phi^{-1}$ is $O(L)$-Lipschitz on $\mathcal{K}$. 

By Lemma~\ref{cor:lip-control}, 
\begin{align*}
u\mapsto \pi_{\mathcal{K}}(u)\cdot \left(c_t^u\circ\phi^{-1} \left(\frac{u}{\pi_{\mathcal{K}}(u)}\right)\right)
\end{align*}
is $O(Ld_x^{4})$-Lipschitz over $\mathcal{K}$, making $c_t^u(u_t(M;K_0))$ $O(Ld_x^{9/2})$-Lipschitz on $\MM_{d_x,d_y,H}$.  . Combining, we have that $e_t$ is $O(Ld_x^{9/2}\tau)$-Lipschitz on $\MM_{d_x,d_y,H}$.  
\end{proof}

\begin{proof} [Proof of \cref{thm:po-main}]
We are ready to prove \cref{thm:po-main}. It is crucial to check the condition in Lemma~\ref{lem:po-mem-mismatch}. Define
\begin{align*}
\|M_t\|_{\star}:=\max_{0\le i\le H}\|M_{t}^{[i]}\|_{1\rightarrow 1}.
\end{align*}
It is easy to check that $\|M\|_{\Sigma}\ge \|M\|_{\star}$. With slight abuse of notation, let $\|M\|_2$ denote the $\ell_2$-norm of the flattened vector in $\mathbb{R}^{d_xd_y(H+1)}$ for $M\in\MM_{d_x,d_y,H}$. Let $\|\cdot\|_{\Sigma}^*$ denote the dual norm of $\|\cdot\|_{\Sigma}$. Lemma~\ref{lem:po-lt-lip} implies that
\cref{alg:gpc-po} guarantees
\begin{align*}
& \quad \|M_{t}-M_{t+1}\|_{\star}\le\|M_{t}-M_{t+1}\|_{\Sigma}\le_{(1)} \sqrt{d_xd_yH}\|M_{t}-M_{t+1}\|_2\le_{(2)} \eta \sqrt{d_xd_yH} \|\partial e_t(M_t)\|_{2}\\
&\le_{(3)}\eta \sqrt{d_xd_yH}\|\partial e_t(M_t)\|_{\Sigma}\le_{(4)} \eta (d_xd_yH)^{3/2} \|\partial e_t(M_t)\|_{\Sigma}^{*}\le_{(5)} \eta(d_xd_yH)^{3/2}Ld_x^{9/2}\tau\\
&\le \eta H^{5/2}Ld_x^{15/2}, 
\end{align*}
for some constant $C$, where $(1)$ and $(3)$ follow from $\ell_1$-norm-$\ell_2$-norm inequality that $\forall v\in\mathbb{R}^n$ $\|v\|_2\le \|v\|_1\le\sqrt{n}\|v\|_2$; $(2)$ follows from the update in Line~\ref{line:po-update} of \cref{alg:gpc-po};
$(4)$ follows from that $\|M\|_{\Sigma}\le d_xd_yH\|M\|_{\Sigma}^*$, and $(5)$ follows from Lemma~\ref{lem:po-lt-lip}, the assumption that $d_x\ge d_y$, and that $H\ge \tau$. 

Moreover, by the standard regret bound of OGD (Theorem 3.1 in \cite{hazan2016introduction}), we have
\begin{align*}
\sum_{t=1}^T e_t(M_t)-\min_{M\in\mathcal{M}}\sum_{t=1}^T e_t(M)\le O\left(\frac{D^2}{\eta}+(Ld_x^{9/2}\tau)^2\eta T\right), 
\end{align*}
where $D$ denotes the diameter of $\MM_{d_x,d_y,H}$ with respect to $\|\cdot\|_{\Sigma}$, which is bounded in this case by 
\begin{align*}
D=\sup_{M\in\mathcal{M}}\left\{ \sum_{i=0}^H\sum_{j\in[d_x],k\in[d_y]}|M^{[i]}_{jk}|\right\}\le 3d_yH.
\end{align*}
Take $\eta=1/(LH^{2}d_x^{4.5}\sqrt{T})$ and recall that $d_y\le d_x$, then we have
\begin{align*}
& \quad \sum_{t=1}^T c_t(y_t,u_t)-\min_{K^{\star}\in \Ksim_\tau(\ML)}\sum_{t=1}^T c_t(y_t(K^{\star}),u_t(K^{\star}))\\
&\le \sum_{t=1}^T c_t(y_t,u_t)-\sum_{t=1}^T c_t(y_t(M_t;K_0),u_t(M_t;K_0))+\sum_{t=1}^T c_t(y_t(M_t;K_0),u_t(M_t;K_0))- \sum_{t=1}^Te_t(M_t)\\
& \quad + \sum_{t=1}^Te_t(M_t)-\min_{M\in\mathcal{M}}\sum_{t=1}^T e_t(M)+1\\
&\le \sum_{t=1}^T e_t(M_t)-\min_{M\in\mathcal{M}}\sum_{t=1}^T e_t(M)+\tilde{O}\left(\eta H^{11/2}L^2d_x^{11} T\right)\\
&\le \tilde{O}\left(\frac{d_x^2H^2}{\eta}+(Ld_x^{9/2}\tau)^2\eta T+\eta H^{11/2}L^2d_x^{11} T\right)\\
&\le \tilde{O}\left(L\tau^{4}d_x^{6.5}\sqrt{T}\right),
\end{align*}
where the first inequality follows from Lemma~\ref{lem:po-approx} by taking $\eps=1$; the second inequality follows from Lemma~\ref{lem:po-mem-mismatch} by taking $\beta=\eta H^{5/2}Ld_x^{15/2}$ and since $c_t(y_t(M_t;K_0),u_t(M_t;K_0))\le e_t(M_t)$ from Observation~\ref{obs:ext}, and the third inequality follows from OGD regret guarantee. 
\end{proof}

%% file: extension.tex
\section{Algorithm for non-Markov policies}
\begin{algorithm}[H]
\caption{$\NMGPCPOS$}
\label{alg:gpc-po-nonmarkov}
\begin{algorithmic}[1]
\REQUIRE Transition matrices $A,B \in \BS^{d_x}$, observation matrix $C \in \BS^{d_y,d_x}$, horizon parameter $H \in \BN$, control parameter $\alpha\in[0,1]$, step size $\eta>0$.  Let $\K$ be defined as in \cref{eq:pair-constrained-set}, and $\pi$ be the Minkowski functional associated with it. Let the transformation map $\phi$ be defined as in \cref{eq:transformation-map}. Denote as $\phi^{-1}$ the inverse of $\phi$ when restricted to $\Delta^{d_x}$.

\STATE Initialize $M_1^{[0:H]}\in \MM := \MM_{d_x,d_y,H}^{+}$ (Definition~\ref{def:learning-class-ldc}). Choose any $K_0\in\mathbb{S}^{d_x,d_y}$. \label{line:nonmarkovchoose-k0}
\STATE Observe initial observation $y_1$.
\FOR{$t=1,\cdots,T$}
    \STATE Compute $o_t := y_t-\sum_{i=1}^{t-1}\bar{\lambda}_{t,i}C[(1-\alpha)A]^{i-1}\alpha B(u_{t-i}-K_0 o_{t-i})$. \label{line:nonmarkovcompute-signal}
    \STATE Let $\tilde{u}_t=M_t(o_{t:t-H})=M_{t}^{[0]} \begin{bmatrix}
o_t\\
\frac{1}{d_x}\mathbf{1}_{d_x}
\end{bmatrix}+\sum_{i=1}^H \bar{\lambda}_{t,i}M_{t}^{[i]} \begin{bmatrix}
o_{t-i}\\
\frac{1}{d_x}\mathbf{1}_{d_x}
\end{bmatrix}$.
    \STATE Choose control $u_t=\phi^{-1}(\phi(\tilde{u}_t)/\pi(\phi(\tilde{u}_t))$.\label{line:nonmarkovcontrol-po}
    \STATE Receive cost $c_t(y_t,u_t)=c_t^y(y_t)+c_t^u(u_t)$, and observe $y_{t+1}$ and $\gamma_t$.
    \STATE Define pseudo loss function 
\begin{align*}
e_t(M) := c_t^y(y_t(M; K_0))+ \pi(\phi(\tilde{u}_t(M;K_0))) \cdot c_t^u(\phi^{-1}(\frac{\phi(\tilde{u}_t(M;K_0))}{\pi(\phi(\tilde{u}_t(M;K_0))})).
\end{align*} \label{line:nonmarkovet-loss}
    \STATE Set $M_{t+1} := \Pi_{\mathcal{M}}^{\|\cdot\|_2}\left[M_t-\eta \cdot \partial e_t(M_t)\right]$, where $\partial e_t(M_t)$ is the subgradient of $e_t$ evaluated at $M_t$. \label{line:nonmarkovpo-update}
\ENDFOR
\end{algorithmic}
\end{algorithm}

\newpage

\section{Proof of Theorem~\ref{thm:po-main-ldc-simplex}}

\subsection{Approximation of simplex-LDC}
\label{sec:approx-simplex-ldc}
In this section, we prove an approximation theorem of simplex-LDC that is analogous to Lemma~\ref{lem:po-approx}. In particular, we show that there exists a convex parametrization that approximates the class of $\tau$-mixing simplex-LDC well. We begin with Observation~\ref{obs:ldc-simplex-evolution}, which describes the system evolution following any simplex-LDC $\pi$.

\begin{observation}[System evolution using LDC-simplex control]
\label{obs:ldc-simplex-evolution}
Given $d_y,d_x,d_x\in\mathbb{N}$, $\tau>0$, a partially observable simplex LDS $\mathcal{L}=(A, B, C, x_1, (\gamma_t)_{t\in\mathbb{N}}, (w_t)_{t\in\BN},(c_t)_{t\in\BN}, \alpha)$, and a LDC-simplex $\pi=(A_{\pi}, B_{\pi}, C_{\pi}, D_{\pi})\in\Pi_{\tau, d_x}(\ML)$, the evolution of the states and observations when exerting controls according to $\pi$ can be described as the follows: 
\begin{align*}
\begin{bmatrix}
x_{t+1}(\pi) \\
s_{t+1}(\pi)
\end{bmatrix}
&=(1-\gamma_t)A_{\pi,\mathrm{cl}}\begin{bmatrix}
x_t(\pi)\\
s_t(\pi)
\end{bmatrix}+\gamma_t \begin{bmatrix}
w_t\\
\frac{1}{d_x}\mathbf{1}_{d_x}
\end{bmatrix}, \\
\begin{bmatrix}
y_{t}(\pi) \\
u_{t}(\pi)
\end{bmatrix}
&=\begin{bmatrix}
C & 0\\
D_{\pi}C & C_{\pi}
\end{bmatrix}\begin{bmatrix}
x_t(\pi)\\
s_t(\pi)
\end{bmatrix},
\end{align*}
where $A_{\pi,\mathrm{cl}}$ is given by \cref{eq:ldc-transition-matrix}.
\end{observation} 

Based on Observation~\ref{obs:ldc-simplex-evolution}, we can unfold the controls and observations by a simplex-LDC $\pi$.

\begin{lemma}[System unfolding using simplex-LDC control]
\label{lem:ldc-simplex-unfolding} 
Consider a partially observable simplex LDS 
$$\ML=(A, B, C, x_1, (\gamma_t)_{t\in\mathbb{N}}, (w_t)_{t\in\BN},(c_t)_{t\in\BN}, \alpha)$$
with state and control dimension $d_x$ and observation dimension $d_y$. For a LDC-simplex $\pi\in\Pi_{\tau,d_x}(\ML)$ of appropriate dimension, the controls and observations $u_t(\pi)$, $y_t(\pi)$ according to $\pi$ unfold as the following expressions:
\begin{align}
\label{eq:control-unfolding}
u_t(\pi)&=\sum_{i=1}^t \lambda_{t,i}C_{\pi,\mathrm{cl},u}A_{\pi,\mathrm{cl}}^{i-1}\begin{bmatrix}
w_{t-i}\\
\frac{1}{d_x}\mathbf{1}_{d_x}
\end{bmatrix},
\end{align}
where we write $w_0=x_1$, $\xi_0=s_1$ for convention and $\Delta_{t,i}, C_{\pi,\mathrm{cl},u}$ are given by
\begin{align*}
C_{\pi,\mathrm{cl},u} =[D_{\pi}C \quad  C_{\pi}]\in\mathbb{R}^{d_x\times(d_x+d_x)},
\end{align*}
and $\lambda_{t,i}$ is defined as in \cref{eq:define-lambdas}.
\end{lemma}
\begin{proof}
By the first equality in Observation~\ref{obs:ldc-simplex-evolution}, we can unroll $(x_t,s_t)$ as the following:
\begin{align*}
\begin{bmatrix}
x_{t}(\pi) \\
s_{t}(\pi)
\end{bmatrix}
&=(1-\gamma_{t-1})A_{\pi,\mathrm{cl}}\begin{bmatrix}
x_{t-1}(\pi)\\
s_{t-1}(\pi)
\end{bmatrix}+\gamma_{t-1} \begin{bmatrix}
w_{t-1}\\
\frac{1}{d_x}\mathbf{1}_{d_x}
\end{bmatrix}\\
&=(1-\gamma_{t-1})(1-\gamma_{t-2})A_{\pi,\mathrm{cl}}^2\begin{bmatrix}
x_{t-2}(\pi)\\
s_{t-2}(\pi)
\end{bmatrix}+(1-\gamma_{t-1})\gamma_{t-2} A_{\pi, \mathrm{cl}}\begin{bmatrix}
w_{t-2}\\
\frac{1}{d_x}\mathbf{1}_{d_x}
\end{bmatrix}+\gamma_{t-1} \begin{bmatrix}
w_{t-1}\\
\frac{1}{d_x}\mathbf{1}_{d_x}
\end{bmatrix}\\
&=\sum_{i=1}^{t} \lambda_{t,i}A_{\pi,\mathrm{cl}}^{i-1}\begin{bmatrix}
w_{t-i}\\
\frac{1}{d_x}\mathbf{1}_{d_x}
\end{bmatrix},
\end{align*}
where $\lambda_{t,i}$ is defined as in \cref{eq:define-lambdas}, $w_0=x_1$. The last step follows by iteratively applying unfolding of $(x_t(\pi), s_t(\pi))$. Then by the second equality in Observation~\ref{obs:ldc-simplex-evolution}, we prove the inequality in Lemma~\ref{lem:ldc-simplex-unfolding}. 
\end{proof}

We consider the same learning policy class and policy as defined in Definition~\ref{def:learning-class} and Definition~\ref{def:learning-policy} and prove an analogous lemma to Lemma~\ref{lem:po-approx}, which states that the learning policy class in Definition~\ref{def:learning-policy} well approximates the class of LDC-simplex with finite mixing time $\Pi_{\tau, d_x}(\ML)$ in Definition~\ref{def:po-comparator-ldc}.

\begin{lemma} [Approximation of LDC-simplex, analogous to Lemma~\ref{lem:po-approx}]
\label{lem:po-approx-ldc}
Consider a partially observable LDS $\ML$ with state and control dimension $d_x$, observation dimension $d_y$. Suppose that the cost functions $c_1,\dots,c_T$ of $\ML$ satisfy \cref{asm:po-convex-loss} with Lipschitz parameter $L>0$. Fix $\tau, \eps>0$ and $d_x\in\BN$. Suppose $H$ satisfies $H\ge 2\tau\lceil \log\frac{4096\tau d_x^{7/2} L T^2}{\epsilon}\rceil$. Then, for any $\pi=(A_{\pi},B_{\pi},C_{\pi},D_{\pi})\in\Pi_{\tau, d_x}(\ML)$, there is some $M\in\MM_{d_x,d_y,d_x,H}$ such that 
\begin{align*}
\left|\sum_{t=1}^T e_t(M)-\sum_{t=1}^T c_t(y_t(\pi), u_t(\pi))\right|\le \eps.  
\end{align*}
\end{lemma}
\begin{proof}
Fix a policy $\pi=(A_{\pi}, B_{\pi}, C_{\pi}, D_{\pi})\in\Pi_{\tau,d_x}(\ML)$. Note that for any fixed $K_0\in\mathcal{K}_{\tau}^{\Delta}(\ML)$, the LDC-simplex policy parameterized by $\pi_0=((1-\alpha)I, \alpha I, 0, K_0)$ satisfies $\pi_0\in \Pi_{\tau, d_x}(\ML)$. 
We show the following two inequalities separately.
\begin{align}
\left|\sum_{t=1}^T c_t(y_t(M;\pi_0),u_t(M;\pi_0))-\sum_{t=1}^T c_t(y_t(\pi), u_t(\pi))\right|\le \frac{\eps}{4} \label{eq:cost-approx-ldc},\\
\left|\sum_{t=1}^T e_t(M)-\sum_{t=1}^T c_t(y_t(M;\pi_0),u_t(M;\pi_0))\right|\le \frac{3\eps}{4} \label{eq:ext-approx-ldc}.
\end{align}
\paragraph{Proof of \cref{eq:cost-approx-ldc}.} 
We start with \cref{eq:cost-approx-ldc}. For each $M$, recall that $\tilde{u}_t(M;\pi_0)$ is the (un-projected but truncated) control parameterized by $M$ using signals generated by the LDC-simplex policy $\pi_0$ before the Minkowski projection. 

Formally, we choose the following parameterization
\begin{align*}
\hat{u}_t(M;\pi_0):=M^{[0]}\begin{bmatrix}
    y_t(\pi_0) \\
    \frac{1}{d_x}\mathbf{1}_{d_x}
\end{bmatrix} +\sum_{i=1}^t\bar{\lambda}_{t,i}M^{[i]}\begin{bmatrix}
    y_{t-i}(\pi_0) \\
    \frac{1}{d_x}\mathbf{1}_{d_x}
\end{bmatrix} ,
\end{align*}
and
\begin{align}
\label{eq:truncated-unprojected-control}
\tilde{u}_t(M;\pi_0):=M^{[0]}\begin{bmatrix}
    y_t(\pi_0) \\
    \frac{1}{d_x}\mathbf{1}_{d_x}
\end{bmatrix} +\sum_{i=1}^H\bar{\lambda}_{t,i}M^{[i]}\begin{bmatrix}
    y_{t-i}(\pi_0) \\
    \frac{1}{d_x}\mathbf{1}_{d_x}
\end{bmatrix} ,
\end{align}
and $u_t(M;\pi_0)$ is obtained by the Minkowski projection of the tuple $(y_t(M;\pi_0), \tilde{u}_t(M;\pi_0))$. We further divide the proof of \cref{eq:cost-approx-ldc} into three steps: (1) we show the un-projected un-truncated control $\hat{u}_t(M;\pi_0)$ is equal to the comparator control $u_t(\pi)$, (2) we show the un-projected un-truncated control $\hat{u}_t(M;\pi_0)$ is close to the un-projected truncated control $\tilde{u}_t(M;\pi_0)$, (3) we show that since $u_t(\pi)$ always satisfies $u_t(\pi)\in\Delta^{d_x}$ and $\tilde{u}_t(M;\pi_0)$ is sufficiently close to $u_t(\pi)$, the projection operator does not change $\tilde{u}_t(M;\pi_0)$ much and therefore preserves the proximity of the projected $u_t(M;\pi_0)$ to $u_t(\pi)$. Then, \cref{eq:cost-approx-ldc} just follows by the regularity assumptions on the system and the Lipschitz assumption on the cost functions. 

\paragraph{Step 1: expressive power of the parameterization.} Our first step is to carefully choose $M$ and $\pi_0$, such that 
$$\exists \pi_0, \forall \pi, \exists M, \forall t, \hat{u}_t(M;\pi_0)=u_t(\pi).$$
In particular, we choose $\pi_0$ to be the linear policy $\begin{bmatrix}
\left(\mathbb{A}_{K_0}^{\alpha}\right)^{j-1} & 0_{d_x\times d_x}\\
0_{d_x\times d_x} & 0_{d_x\times d_x}
\end{bmatrix}$ where $\mathbb{A}_{K_0}^{\alpha}=(1-\alpha)A+\alpha B K_0 C$, and $M$ to be
$$
M^{[0]}=\begin{bmatrix}
D_{\pi} & C_{\pi}
\end{bmatrix}, \quad M^{[i]}=C_{\pi,cl,u} A_{\pi,cl}^{i-1} P_{\pi,\pi_0}, \quad \forall 1\le i \le H,
$$
where $P_{\pi,\pi_0} =\begin{bmatrix}
\alpha B(D_{\pi}-K_0) & \alpha B C_{\pi}\\
B_{\pi} & A_{\pi}-I
\end{bmatrix}$. Recall that from Lemma~\ref{lem:ldc-simplex-unfolding}, the target control $u_t(\pi)$ can be written as the following:
\begin{align*}
u_t(\pi)=\sum_{i=1}^t \lambda_{t,i} C_{\pi,cl,u} (A_{\pi,cl})^{i-1} \begin{bmatrix}
    w_{t-i} \\
    \frac{1}{d_x}\mathbf{1}_{d_x}
\end{bmatrix},
\end{align*}
and we would like to show $\hat{u}_t(M;\pi_0)$ is exactly equal to it.

To this end, we analyze the two terms of $\hat{u}_t(M;\pi_0)$ separately. First, since $\pi_0$ is effectively a Markov policy, we can unfold the signals $y_t(\pi_0)$ generated by $\pi_0$ similarly as in the Markov comparator class setting, i.e.
\begin{align*}
y_t(\pi_0)=\sum_{i=1}^t\lambda_{t,i}C(\mathbb{A}_{K_0}^{\alpha})^{i-1}w_{t-i}, \quad \mathbb{A}_{K_0}^{\alpha}=(1-\alpha)A+\alpha BK_0C.
\end{align*}
Using the unfolding of the signals, we can expand the two terms of $\hat{u}_t(M;\pi_0)$ as the following:
\begin{align*}
    M^{[0]}\begin{bmatrix}
    y_t(\pi_0) \\
    \frac{1}{d_x}\mathbf{1}_{d_x}
\end{bmatrix}&=\begin{bmatrix}
    D_{\pi} & C_{\pi}
\end{bmatrix} \begin{bmatrix}
    \sum_{i=1}^t \lambda_{t,i}C (\mathbb{A}_{K_0}^{\alpha})^{i-1} w_{t-i} \\ \sum_{i=1}^t \lambda_{t,i} \frac{1}{d_x}\mathbf{1}_{d_x}
\end{bmatrix}\\
&=\sum_{i=1}^t \lambda_{t,i} \begin{bmatrix}
    D_{\pi} & C_{\pi}
\end{bmatrix} \begin{bmatrix}
    C & 0 \\ 0 & I_{d_x\times d_x}
\end{bmatrix} \begin{bmatrix}
    \mathbb{A}_{K_0}^{\alpha} & 0 \\ 0 & I_{d_x\times d_x}
\end{bmatrix}^{i-1} \begin{bmatrix}
    w_{t-i} \\ \frac{1}{d_x}\mathbf{1}_{d_x}
\end{bmatrix}\\
&=\sum_{i=1}^t \lambda_{t,i} \begin{bmatrix}
    D_{\pi} C & C_{\pi}
\end{bmatrix} \begin{bmatrix}
    \mathbb{A}_{K_0}^{\alpha} & 0 \\ 0 & I_{d_x\times d_x}
\end{bmatrix}^{i-1} \begin{bmatrix}
    w_{t-i} \\ \frac{1}{d_x}\mathbf{1}_{d_x}
\end{bmatrix}\\
&=\sum_{i=1}^t \lambda_{t,i} C_{\pi,cl,u} \begin{bmatrix}
    \mathbb{A}_{K_0}^{\alpha} & 0 \\ 0 & I_{d_x\times d_x}
\end{bmatrix}^{i-1} \begin{bmatrix}
    w_{t-i} \\ \frac{1}{d_x}\mathbf{1}_{d_x}
\end{bmatrix},
\end{align*}
where in the first equality we used the identity that $\forall t\ge 1$, $\sum_{i=1}^{t}\lambda_{t,i}=1$.
For the second part, we have that
\begin{align*}    
& \quad \sum_{i=1}^t\bar{\lambda}_{t,i}M^{[i]}\begin{bmatrix}
    y_{t-i}(\pi_0) \\
    \frac{1}{d_x}\mathbf{1}_{d_x}
\end{bmatrix}\\
&= \sum_{i=1}^t \bar{\lambda}_{t,i} C_{\pi,cl,u} A_{\pi,cl}^{i-1} P_{\pi,\pi_0} \begin{bmatrix}
    \sum_{j=1}^{t-i} \lambda_{t-i,j}C (\mathbb{A}_{K_0}^{\alpha})^{j-1} w_{t-i-j} \\ \sum_{j=1}^{t-i} \lambda_{t-i,j} \frac{1}{d_x}\mathbf{1}_{d_x}
\end{bmatrix}\\
&=\sum_{i=1}^t \sum_{j=i+1}^t \bar{\lambda}_{t,i} \lambda_{t-i,j-i} C_{\pi,cl,u} A_{\pi,cl}^{i-1} P_{\pi,\pi_0} \begin{bmatrix}
    C & 0 \\ 0 & I_{d_x\times d_x}
\end{bmatrix} \begin{bmatrix}
    \mathbb{A}_{K_0}^{\alpha} & 0 \\ 0 & I_{d_x\times d_x}
\end{bmatrix}^{j-i-1} \begin{bmatrix}
    w_{t-j} \\ \frac{1}{d_x}\mathbf{1}_{d_x}
\end{bmatrix}\\
&= \sum_{i=1}^t \sum_{j=i+1}^t \lambda_{t,j} C_{\pi,cl,u} A_{\pi,cl}^{i-1} P_{\pi,\pi_0} \begin{bmatrix}
    C & 0 \\ 0 & I_{d_x\times d_x}
\end{bmatrix} \begin{bmatrix}
    \mathbb{A}_{K_0}^{\alpha} & 0 \\ 0 & I_{d_x\times d_x}
\end{bmatrix}^{j-i-1} \begin{bmatrix}
    w_{t-j} \\ \frac{1}{d_x}\mathbf{1}_{d_x}
\end{bmatrix}\\
&=\sum_{j=2}^t \sum_{i=1}^{j-1} \lambda_{t,j} C_{\pi,cl,u} A_{\pi,cl}^{i-1} P_{\pi,\pi_0} \begin{bmatrix}
    C & 0 \\ 0 & I_{d_x\times d_x}
\end{bmatrix} \begin{bmatrix}
    \mathbb{A}_{K_0}^{\alpha} & 0 \\ 0 & I_{d_x\times d_x}
\end{bmatrix}^{j-i-1} \begin{bmatrix}
    w_{t-j} \\ \frac{1}{d_x}\mathbf{1}_{d_x}
\end{bmatrix},
\end{align*}
where the second to the last equality follows from that $\bar{\lambda}_{t,i}\lambda_{t-i, j-i}=\lambda_{t,j}$ as shown in \cref{eq:lambda-property}. Here we recall that
$$
A_{\pi,cl}=\begin{bmatrix}
    (1-\alpha)A+\alpha BD_{\pi}C & \alpha B C_{\pi} \\ B_{\pi} C & A_{\pi}
\end{bmatrix}, \quad P_{\pi,\pi_0} \begin{bmatrix}
    C & 0 \\ 0 & I
\end{bmatrix}= \begin{bmatrix}
    \alpha B(D_{\pi} K_0)C & \alpha B C_{\pi} \\ B_{\pi} C & A_{\pi} - I_{d_x \times d_x}
\end{bmatrix},
$$
which yields the nice property
$$
A_{\pi,cl}-\begin{bmatrix}
    \mathbb{A}_{K_0}^{\alpha} & 0 \\ 0 & I_{d_x\times d_x}
\end{bmatrix}=P_{\pi,\pi_0} \begin{bmatrix}
    C & 0 \\ 0 & I_{d_x \times d_x}
\end{bmatrix}.
$$
Now, we can continue rewriting the previous equation
\begin{align*}
    & \quad \sum_{i=1}^t\bar{\lambda}_{t,i}M^{[i]}\begin{bmatrix}
    y_{t-i}(\pi_0) \\
    \frac{1}{d_x}\mathbf{1}_{d_x}
\end{bmatrix}\\
&= \sum_{j=2}^t \sum_{i=1}^{j-1} \lambda_{t,j} C_{\pi,cl,u} A_{\pi,cl}^{i-1} P_{\pi,\pi_0} \begin{bmatrix}
    C & 0 \\ 0 & I_{d_x\times d_x}
\end{bmatrix} \begin{bmatrix}
    \mathbb{A}_{K_0}^{\alpha} & 0 \\ 0 & I_{d_x\times d_x}
\end{bmatrix}^{j-i-1} \begin{bmatrix}
    w_{t-j} \\ \frac{1}{d_x}\mathbf{1}_{d_x}
\end{bmatrix}\\
&=\sum_{j=2}^t \lambda_{t,j} C_{\pi,cl,u} \left( \sum_{i=1}^{j-1} A_{\pi,cl}^{i-1} \left(A_{\pi,cl}-\begin{bmatrix}
    \mathbb{A}_{K_0}^{\alpha} & 0 \\ 0 & I_{d_x\times d_x}
\end{bmatrix}\right) \begin{bmatrix}
    \mathbb{A}_{K_0}^{\alpha} & 0 \\ 0 & I_{d_x\times d_x}
\end{bmatrix}^{j-i-1} \right)\begin{bmatrix}
    w_{t-j} \\ \frac{1}{d_x}\mathbf{1}_{d_x}
\end{bmatrix}\\
&= \sum_{j=2}^t \lambda_{t,j} C_{\pi,cl,u} \left( A_{\pi,cl}^{j-1} - \begin{bmatrix}
    \mathbb{A}_{K_0}^{\alpha} & 0 \\ 0 & I_{d_x\times d_x}
\end{bmatrix}^{j-1} \right)\begin{bmatrix}
    w_{t-j} \\ \frac{1}{d_x}\mathbf{1}_{d_x}
\end{bmatrix}\\
&=\sum_{j=1}^t \lambda_{t,j} C_{\pi,cl,u} \left( A_{\pi,cl}^{j-1} - \begin{bmatrix}
    \mathbb{A}_{K_0}^{\alpha} & 0 \\ 0 & I_{d_x\times d_x}
\end{bmatrix}^{j-1} \right)\begin{bmatrix}
    w_{t-j} \\ \frac{1}{d_x}\mathbf{1}_{d_x}
\end{bmatrix}.
\end{align*}
Combining the two parts, we have the desired equation
$$
\hat{u}_t(M;\pi_0):=M^{[0]}\begin{bmatrix}
    y_t(\pi_0) \\
    \frac{1}{d_x}\mathbf{1}_{d_x}
\end{bmatrix} +\sum_{i=1}^t\bar{\lambda}_{t,i}M^{[i]}\begin{bmatrix}
    y_{t-i}(\pi_0) \\
    \frac{1}{d_x}\mathbf{1}_{d_x}
\end{bmatrix} = \sum_{i=1}^t \lambda_{t,i} C_{\pi,cl,u} (A_{\pi,cl})^{i-1} \begin{bmatrix}
    w_{t-i} \\
    \frac{1}{d_x}\mathbf{1}_{d_x}
\end{bmatrix} =u_t(\pi).
$$

\paragraph{Step 2: low truncation error.}
The next step is to show that with the mixing time assumption on $A_{\pi,\mathrm{cl}}$ in Definition~\ref{def:mixing-simplex-ldc}, the truncation loss $\|\hat{u}_t(M;\pi_0)-\tilde{u}_t(M;\pi_0)\|_2$ is small, where
\begin{align*}
\tilde{u}_t(M;\pi_0):=M^{[0]}\begin{bmatrix}
    y_t(\pi_0) \\
    \frac{1}{d_x}\mathbf{1}_{d_x}
\end{bmatrix} +\sum_{i=1}^H\bar{\lambda}_{t,i}M^{[i]}\begin{bmatrix}
    y_{t-i}(\pi_0) \\
    \frac{1}{d_x}\mathbf{1}_{d_x}
\end{bmatrix} 
\end{align*}
is the truncated but unprojected control parameterized by $M$ (\cref{eq:truncated-unprojected-control}) operating on the Markov signals $y_t(\pi_0)$.

By definition, we unfold $\hat{u}_t(M;\pi_0)$ and $\tilde{u}_t(M;\pi_0)$ and get
\begin{align*}
\|\hat{u}_t(M;\pi_0)-\tilde{u}_t(M;\pi_0)\|&=\left\|\sum_{i=H+1}^t\bar{\lambda}_{t,i}M^{[i]}\begin{bmatrix}
    y_{t-i}(\pi_0) \\
    \frac{1}{d_x}\mathbf{1}_{d_x}
\end{bmatrix} \right\|\\
&=\left\|\sum_{i=H+1}^t \bar{\lambda}_{t,i}C_{\pi,\mathrm{cl},u}A_{\pi,\mathrm{cl}}^{i-1}\begin{bmatrix}
 \alpha B (D_{\pi}-K_0) & \alpha B C_{\pi} \\
    B_{\pi} & A_{\pi}-I
\end{bmatrix}
\begin{bmatrix}
    y_{t-i}(\pi_0) \\
    \frac{1}{d_x}\mathbf{1}_{d_x}
\end{bmatrix}\right\|.
\end{align*}
We have the following equivalent form for any $v_y\in \Delta^{d_y}, v_x\in \Delta^{d_x}$:
\begin{align*}
\begin{bmatrix}
 \alpha B (D_{\pi}-K_0) & \alpha B C_{\pi} \\
    B_{\pi} & A_{\pi}-I
\end{bmatrix}
\begin{bmatrix}
    v_y \\
    v_x
\end{bmatrix}&=\left(\begin{bmatrix}
    B(D_{\pi}v_y+C_{\pi}v_x) \\ B_{\pi}v_y+A_{\pi}v_x
\end{bmatrix}-\begin{bmatrix}
    K_0 v_y \\ v_x
\end{bmatrix}\right) \\
& \quad + (1-\alpha)\left(\begin{bmatrix}
    B(D_{\pi}v_y+C_{\pi}v_x) \\ \frac{1}{d_x}\mathbf{1}_{d_x}
\end{bmatrix}-\begin{bmatrix}
    K_0 v_y \\ \frac{1}{d_x}\mathbf{1}_{d_x}
\end{bmatrix}\right).
\end{align*}
Rewrite as following for simplicity of notation:
\begin{align*}
v_1=\begin{bmatrix}
    B(D_{\pi}v_y+C_{\pi}v_x) \\ B_{\pi}v_y+A_{\pi}v_x
\end{bmatrix}, \quad v_2=\begin{bmatrix}
    K_0 v_y \\ v_x
\end{bmatrix}\\
v_3=\begin{bmatrix}
    B(D_{\pi}v_y+C_{\pi}v_x) \\ \frac{1}{d_x}\mathbf{1}_{d_x}
\end{bmatrix}, \quad v_4 = \begin{bmatrix}
    K_0 v_y \\ \frac{1}{d_x}\mathbf{1}_{d_x}
\end{bmatrix}.
\end{align*}
Note that $v_1, v_2, v_3, v_4\in\Delta^{d_x}\times \Delta^{d_x}$. Recall the mixing-time assumption on $A_{\pi,\mathrm{cl}}$ as in Definition~\ref{def:mixing-simplex-ldc}, denote $x^{\star}, s^{\star}\in\Delta^{d_x}$ as the unique pair of vectors such that 
\begin{align*}
t^{\mathrm{mix}}(A_{\pi,\mathrm{cl}})=\min_{t\in\mathbb{N}}\left\{t: \quad \forall t'\ge t, \sup_{p\in\Delta^{d_x},q\in\Delta^{d_x}}\|A_{\pi,\mathrm{cl}}^t(p,q)-(x^{\star},s^{\star})\|_1\le \frac{1}{4}\right\}\le \tau.
\end{align*}
For simplicity, write $p^{\star}=(x^{\star}, s^{\star})\in\mathbb{R}^{2d_x}$ to be the concatenated vector.
Then, we can further bound $\|\hat{u}_t(M;\pi_0)-\tilde{u}_t(M;\pi_0)\|$ by
\begin{align*}
&\quad \|\hat{u}_t(M;\pi_0)-\tilde{u}_t(M;\pi_0)\|\\
&=\left\|\sum_{i=H+1}^t \bar{\lambda}_{t,i}C_{\pi,\mathrm{cl},u}A_{\pi,\mathrm{cl}}^{i-1}[(v_1-v_2)+(1-\alpha)(v_3-v_4)]\right\|\\
&\le \left\|\sum_{i=H+1}^t \bar{\lambda}_{t,i}C_{\pi,\mathrm{cl},u}A_{\pi,\mathrm{cl}}^{i-1}(v_1-v_2)\right\| + (1-\alpha)\left\|\sum_{i=H+1}^t \bar{\lambda}_{t,i}C_{\pi,\mathrm{cl},u}A_{\pi,\mathrm{cl}}^{i-1}(v_3-v_4)\right\|\\
&\le \left\|\sum_{i=H+1}^t \bar{\lambda}_{t,i}C_{\pi,\mathrm{cl},u}A_{\pi,\mathrm{cl}}^{i-1}(v_1-p^{\star})\right\| + \left\|\sum_{i=H+1}^t \bar{\lambda}_{t,i}C_{\pi,\mathrm{cl},u}A_{\pi,\mathrm{cl}}^{i-1}(v_2-p^{\star})\right\|\\
&\quad + (1-\alpha)\left\|\sum_{i=H+1}^t \bar{\lambda}_{t,i}C_{\pi,\mathrm{cl},u}A_{\pi,\mathrm{cl}}^{i-1}(v_3-p^{\star})\right\|+ (1-\alpha)\left\|\sum_{i=H+1}^t \bar{\lambda}_{t,i}C_{\pi,\mathrm{cl},u}A_{\pi,\mathrm{cl}}^{i-1}(v_4-p^{\star})\right\|.
\end{align*}
By the mixing-time assumption, choice of $H\ge 2\tau\lceil \log\frac{4096\tau d_x^{7/2} L T^2}{\epsilon}\rceil$ and Lemma 18 in \citep{golowich2024online}, we can further bound 
\begin{align*}
\|\hat{u}_t(M;\pi_0)-\tilde{u}_t(M;\pi_0)\|&\le 4 \sum_{i=H}^{\infty}2^{-i/\tau}\le \frac{\eps}{512d_x^{7/2}LT^2}
\end{align*}
Since in the previous step we have shown that $\hat{u}_t(M;\pi_0)=u_t(\pi)$, the above bound directly translates to
\begin{align*}
\|\tilde{u}_t(M;\pi_0)-u_t(\pi)\|&\le 4 \sum_{i=H}^{\infty}2^{-i/\tau}\le \frac{\eps}{512d_x^{7/2}LT^2}
\end{align*}

\paragraph{Step 3: low projection error.} By the previous step and the Lipschitz assumptions of the operators $\phi$ and $\pi$, we have that since $u_t(\pi)\in\Delta^{d_x}$,
\begin{align*}
\|u_t(M;\pi_0)-u_t(\pi)\|_1&=\left\|\phi^{-1}(\frac{\phi( \tilde{u}_t(M;\pi_0))}{\pi(\phi( \tilde{u}_t(M;\pi_0))})-\phi^{-1}(\frac{\phi( u_t(\pi))}{\pi(\phi( u_t(\pi))})\right\|_1\\
&\le 2\cdot \left\|\frac{\phi( \tilde{u}_t(M;\pi_0))}{\pi(\phi( \tilde{u}_t(M;\pi_0))}-\frac{\phi( u_t(\pi))}{\pi(\phi( u_t(\pi))}\right\|_1\\
&\le 64 d_x^{7/2}\cdot \|\tilde{u}_t(M;\pi_0)-u_t(\pi)\|_1\\
&\le 64 d_x^{7/2}\cdot \frac{\eps}{512d_x^{7/2}LT^2}\\
&= \frac{\eps}{8LT^2}.
\end{align*}
Note that the observation $y_t(M;\pi_0), y_t(\pi)$ can unfold as follows as in \cref{eq:yt-unfold-po} and thus
\begin{align*}
\|y_t(M;\pi_0)-y_t(\pi)\|_1&=\left\|\sum_{i=1}^{t-1}\bar{\lambda}_{t,i}C((1-\alpha)A)^{i-1}\alpha B(u_{t-i}(M;\pi_0)-u_{t-i}(\pi))\right\|_1\\
&\le \frac{\eps}{8LT^2} \sum_{i=1}^{t-1}\bar{\lambda}_{t,i} \le  \frac{\eps}{8LT^2} \cdot T. 
\end{align*}
By the Lipschitz assumption on the cost function $c_t$ in Assumption~\ref{asm:po-convex-loss}, we have that
\begin{align*}
&\quad \left|\sum_{t=1}^T c_t(y_t(M;\pi_0),u_t(M;\pi_0))-\sum_{t=1}^T c_t(y_t(\pi), u_t(\pi))\right|\\
&\le \sum_{t=1}^T |c_t^y(y_t(M;\pi_0))-c_t^y(y_t(\pi))|+\sum_{t=1}^T |c_t^u(u_t(M;\pi_0))-c_t^u(u_t(\pi))|\\
&\le L\cdot \left(\sum_{t=1}^T \|y_t(M;\pi_0)-y_t(\pi)\|_1+\sum_{t=1}^T \|u_t(M;\pi_0)-u_t(\pi)\|_1\right)\\
&\le 2LT \cdot \frac{\eps}{8LT}\\
&= \frac{\eps}{4}.
\end{align*}
Thus, we have proved \cref{eq:cost-approx-ldc}. We move on to prove \cref{eq:ext-approx-ldc}. 

\paragraph{Proof of \cref{eq:ext-approx-ldc}.} The proof follows similarly to the proof of \cref{eq:ext-approx}. By the definition of $e_t(M)$ in \cref{eq:et-def}, we have
\begin{align*}
&\quad \left|\sum_{t=1}^T e_t(M)-\sum_{t=1}^T c_t(y_t(M;\pi_0),u_t(M;\pi_0))\right|\\
&=\left|\sum_{t=1}^T \pi_{\mathcal{K}}(\phi(\tilde{u}_t(M;\pi_0)))\cdot c_t^u\left(\phi^{-1}\left(\frac{\phi(\tilde{u}_t(M;\pi_0)}{\pi_{\mathcal{K}}(\phi(\tilde{u}_t(M;K_0)))}\right)\right)-c_t^u(u_t(M;\pi_0))\right|\\
&\le \left|\sum_{t=1}^T (\pi_{\mathcal{K}}(\phi(\tilde{u}_t(M;\pi_0)))-1)\cdot c_t^u\left(\phi^{-1}\left(\frac{\phi(\tilde{u}_t(M;\pi_0)}{\pi_{\mathcal{K}}(\phi(\tilde{u}_t(M;K_0)))}\right)\right)\right|\\
&\quad +\left|\sum_{t=1}^T c_t^u\left(\phi^{-1}\left(\frac{\phi(\tilde{u}_t(M;\pi_0)}{\pi_{\mathcal{K}}(\phi(\tilde{u}_t(M;K_0)))}\right)\right)-c_t^u(u_t(M;\pi_0))\right|.
\end{align*}
To bound the first term, note that
\begin{align*}
|\pi_{\mathcal{K}}(\phi(\tilde{u}_t(M;\pi_0)))-1|&=|\pi_{\mathcal{K}}(\phi(\tilde{u}_t(M;\pi_0)))-\pi_{\mathcal{K}}(\phi(u_t(\pi)))|\\
&\le 2d_x\cdot \|\tilde{u}_t(M;\pi_0)-u_t(\pi)\|_1\\
&\le \frac{\eps}{256d_x^{5/2}LT^2},
\end{align*}
and therefore by Lipschitz assumption of $c_t$,
\begin{align*}
\left|\sum_{t=1}^T (\pi_{\mathcal{K}}(\phi(\tilde{u}_t(M;\pi_0)))-1)\cdot c_t^u\left(\phi^{-1}\left(\frac{\phi(\tilde{u}_t(M;\pi_0)}{\pi_{\mathcal{K}}(\phi(\tilde{u}_t(M;K_0)))}\right)\right)\right|&\le \frac{\eps}{256d_x^{5/2}T}.
\end{align*}
To bound the second term, we have
\begin{align*}
& \quad \left|\sum_{t=1}^T c_t^u\left(\phi^{-1}\left(\frac{\phi(\tilde{u}_t(M;\pi_0)}{\pi_{\mathcal{K}}(\phi(\tilde{u}_t(M;K_0)))}\right)\right)-c_t^u(u_t(M;\pi_0))\right|\\
&\le L \sum_{t=1}^T \left\|\phi^{-1}\left(\frac{\phi(\tilde{u}_t(M;\pi_0)}{\pi_{\mathcal{K}}(\phi(\tilde{u}_t(M;K_0)))}\right)-u_t(M;\pi_0)\right\|_1\\
&\le 2L\sum_{t=1}^T \left\|\frac{\phi(\tilde{u}_t(M;\pi_0)}{\pi_{\mathcal{K}}(\phi(\tilde{u}_t(M;K_0)))}-\phi(u_t(M;\pi_0))\right\|_1\\
&\le 2L\sum_{t=1}^{T}  \frac{\|\phi(\tilde{u}_t(M;\pi_0))-\phi(u_t(M;\pi_0))\|_1}{\pi_{\mathcal{K}}(\phi(\tilde{u}_t(M;\pi_0)))} \\
&\quad + 2L\sum_{t=1}^{T}  \left(1-\frac{1}{\pi_{\mathcal{K}}(\phi(\tilde{u}_t(M;\pi_0)))}\right)\cdot \|\phi(u_t(M;\pi_0))\|_1\\
&\le 2L\sum_{t=1}^T\|\tilde{u}_t(M;\pi_0)-u_t(M;\pi_0)\|_1+2L\sum_{t=1}^T (\pi_{\mathcal{K}}(\phi(\tilde{u}_t(M;\pi_0)))-1)\\
&\le 2LT\left(\frac{\eps}{8LT^2}+\frac{\eps}{512d_x^{7/2}LT^2}+\frac{\eps}{256d_x^{5/2}LT^2}\right)\le \frac{\eps}{2T}.
\end{align*}
Combining, we have that for all $T\ge 1$,
\begin{align*}
\left|\sum_{t=1}^T e_t(M)-\sum_{t=1}^T c_t(y_t(M;\pi_0),u_t(M;\pi_0))\right|\le \frac{3\eps}{4},
\end{align*}
and conclude the proof. 
\end{proof}

\subsection{Bounding the memory mismatch error}
\label{sec:mem-mismatch-po-nonmarkov}
\begin{lemma} [Memory mismatch error, non-Markov Policy]
\label{lem:nonmarkovpo-mem-mismatch}
Suppose that $(c_t)_t$ satisfy \cref{asm:po-convex-loss} with Lipschitz parameter $L$. Let $\tau,\beta > 0$, and suppose that $\Pi_{\tau}(\ML)$ is nonempty. Consider the execution of $\NMGPCPOS$ (\cref{alg:gpc-po-nonmarkov}) on $\ML$. If the iterates $(M_t\^{0:H})_{t \in [T]}$ satisfy
 \begin{align}
   % \gpcnorm{(p_t, M_t\^{1:H}) - (p_{t+1}, M_{t+1}\^{1:H})} \leq \beta
   \max_{1 \leq t \leq T-1} \max_{i \in [H]} \oneonenorm{M_t\^i - M_{t+1}\^i} \leq \beta,
   \label{eq:ptmt-change-po-nonmarkov}
 \end{align}
then for each $t \in [T]$, the loss function $\ell_t$ computed at time step $t$ satisfies
    \begin{align}
| c_t(y_t(M_t;K_0),u_t(M_t;K_0)) - c_t(y_t, u_t)| &\leq O\left(L d_x^{11/2} \beta\log^2(1/\beta)H\tau^2\right)\nonumber.
    \end{align}
\iffalse
Assume that \cref{alg:gpc-po} satisfies
\begin{align*}
\max_{1\le t\le T-1}\max_{0\le i\le H} \|M_{t}^{[i]}-M_{t+1}^{[i]}\|_{1\rightarrow1}\le \beta,
\end{align*}
then we have
\begin{align*}
\left|\sum_{t=1}^T \ell_t(M_t)-\sum_{t=1}^T c_t(y_t,u_t)\right|\le O(L\tau^2\beta\log^2(1/\beta)HT).
\end{align*}
\fi
\end{lemma}

\begin{proof} 
Fix $t \in [T]$. First, by \cref{eq:yt-unfold-po}, we have that for all $M\in\MM_{d_x,d_y,H}^{+}$, %\noah{instances of $y_t(M)$ should be $y_t(M; K_0)$ (and same for $u_t(M)$)?}
\begin{align}
\label{eq:yt-diff-nonmarkov}
y_t-y_t(M;K_0) &= \sum_{i=1}^{t-1}\bar{\lambda}_{t,i}C ((1-\alpha)A)^{i-1}\alpha B(u_{t-i}-u_{t-i}(M;K_0)). 
\end{align}
We consider two cases of the mixing time $\tau_A$ of $A$. Let $\tau_A=t^{\mathrm{mix}}(A)$, and $t_0=\lceil \tau_A\log(2/\beta)\rceil$. 

\paragraph{Case 1: $\tau_A\le 4\tau$.} Let $p_{A}^{\star}$ satisfy $p_{A}^{\star}=Ap_{A}^{\star}$ be a stationary distribution of $A$, then by adding and subtracting $p_{A}^{\star}$, we have $\forall M\in\MM_{d_x,d_y,H}^{+}$,
\begin{align*}
&\left\|\sum_{i=t_0+1}^{t-1}\bar{\lambda}_{t,i}[(1-\alpha)A]^{i-1}\alpha B(u_{t-i}-u_{t-i}(M;K_0))\right\|_1\\
&\le \left\|\alpha \sum_{i=t_0+1}^{t-1}\bar{\lambda}_{t,i}[(1-\alpha)A]^{i-1}(Bu_{t-i}-p_{A}^{\star})\right\|_1+ \left\|\alpha\sum_{i=t_0+1}^{t-1}\bar{\lambda}_{t,i}[(1-\alpha)A]^{i-1}(Bu_{t-i}(M;K_0)-p_{A}^{\star})\right\|_1\\
&\le \alpha \sum_{i=t_0+1}^{t-1}\bar{\lambda}_{t,i}(1-\alpha)^{i-1}\|A^{i-1}(Bu_{t-i}-p_{A}^{\star})\|_1+ \alpha\sum_{i=t_0+1}^{t-1}\bar{\lambda}_{t,i}(1-\alpha)^{i-1}\|A^{i-1}(Bu_{t-i}(M;K_0)-p_{A}^{\star})\|_1\\
&\le 2 \sum_{i=t_0+1}^{t-1} \frac{1}{2^{\lfloor (i-1)/\tau_A\rfloor}}\le 2\tau_A\beta\sum_{i=0}^{\infty}\frac{1}{2^i}\le C\tau\beta\log(1/\beta).
\end{align*}
Thus, it suffices to show a bound for the first $t_0$ terms. We have that
\begin{align*}
\|\tilde{u}_{t-i}-\tilde{u}_{t-i}(M_t;K_0)\|_1&\le \|(M_{t-i}^{[0]}-M_{t}^{[0]})y_{t-i}(K_0)\|_1+\sum_{j=1}^{H}\bar{\lambda}_{t-i,j}\|(M_{t-i}^{[j]}-M_{t}^{[j]})y_{t-i-j}(K_0)\|_1\\
&\le (H+1)i\beta\le 2Hi\beta,
\end{align*}
since $y_t\in\Delta^{d_y}$, $\forall t$, and $\max_{0\le j\le H}\|M_{t-i}^{[j]}-M_{t}^{[j]}\|_{1\rightarrow 1}\le i\beta$ by the assumption. To bound the $\|u_{t-i}-u_{t-i}(M_t;K_0)\|_1$, we note that by the non-expanding condition established in Lemma~\ref{thm:proj-lip}, we have that
\begin{align*}
\|u_{t-i}-u_{t-i}(M_t;K_0)\|_1&\le 2\sqrt{d_x}\left\|u_{t-i}-u_{t-i}(M_t;K_0)\right\|_2\\
&\le C d_x^{7/2}\|\tilde{u}_{t-i}-\tilde{u}_{t-i}(M_t;K_0)\|_1\\
&\le Cd_x^{7/2}Hi\beta.
\end{align*}
Thus, we can bound $\|y_t-y_t(M_t;K_0)\|_1$ as the following: 
\begin{align*}
\|y_t-y_t(M_t;K_0)\|_1&\le \left\|\sum_{i=1}^{t_0}\bar{\lambda}_{t,i}C[(1-\alpha)A]^{i-1}\alpha B(u_{t-i}-u_{t-i}(M))\right\|_1+C\tau\beta\log(1/\beta)\\
&\le Cd_x^{7/2}\beta H t_0^2 +C\tau\beta\log(1/\beta)\\
&\le Cd_x^{7/2}\beta \log^2(1/\beta) H\tau^2
\end{align*}
for some constant $C$.

\paragraph{Case 2: $\tau_A>4\tau$.} 

Similarly, we will give a lower bound on $\alpha$. Recall that 
$$\tau=t^{\mathrm{mix}}(A_{\pi,\mathrm{cl}})=\min_{t\in\mathbb{N}}\left\{t: \quad \forall t'\ge t, \sup_{p\in\Delta^{d_x},q\in\Delta^{d_x}}\|A_{\pi,\mathrm{cl}}^t(p,q)-(x,s)\|_1\le \frac{1}{4}\right\},$$
where $A_{\pi,\mathrm{cl}}=\begin{bmatrix}
(1-\alpha)A+\alpha BD_{\pi}C & \alpha BC_{\pi} \\
 B_{\pi}C & A_{\pi}
\end{bmatrix}$. Suppose (for the purpose of contradiction) that $\alpha\le 1/(256d_x\tau)$. Denote 
$A_{compare}=\begin{bmatrix}
A & 0 \\
 B_{\pi}C & A_{\pi}
\end{bmatrix}$. We have that
$$
A_{compare}-A_{\pi,\mathrm{cl}}=\begin{bmatrix}
-\alpha(BD_{\pi}C -A) & \alpha B C_{\pi} \\
 0 & 0
\end{bmatrix}.
$$
For any $v_1,v_2\in \Delta^{d_x}$, we have that
\begin{align*}
 A_{compare}^t \begin{bmatrix}
v_1 \\
v_2
\end{bmatrix}-   A_{\pi,\mathrm{cl}}^t \begin{bmatrix}
v_1 \\
v_2
\end{bmatrix} &=\sum_{i=1}^t A_{\pi,\mathrm{cl}}^{t-i}(A_{compare}-A_{\pi,\mathrm{cl}}) A_{compare}^{i-1} \begin{bmatrix}
v_1 \\
v_2
\end{bmatrix}\\
&=\alpha\sum_{i=1}^t A_{\pi,\mathrm{cl}}^{t-i} \begin{bmatrix}
A-BD_{\pi}C & B C_{\pi} \\
 0 & 0
\end{bmatrix}A_{compare}^{i-1} \begin{bmatrix}
v_1 \\
v_2
\end{bmatrix}.
\end{align*}
Notice that $A_{compare}^{i-1}$ is always of form $\begin{bmatrix}
A & 0 \\
 B^* & C^*
\end{bmatrix}$ for $i>1$, where $B^*+C^*\in \BS^{d_x,d_x}$ and both $B^*, C^*$ are non-negative, meaning that $A_{compare}^{i-1}$ always maps $\Delta^{d_x} \times \Delta^{d_x}$ to itself. Therefore $A_{compare}^{i-1} \begin{bmatrix}
v_1 \\
v_2
\end{bmatrix}\in \Delta^{d_x} \times \Delta^{d_x}$ for any $i\ge 1$.

The matrix $\begin{bmatrix}
A-BD_{\pi}C & B C_{\pi} \\
 0 & 0
\end{bmatrix}$ always maps $\Delta^{d_x} \times \Delta^{d_x}$ to $\Delta^{d_x} \times \Delta^{d_x}-\Delta^{d_x} \times \Delta^{d_x}$. Because $A_{\pi,\mathrm{cl}}$ also maps $\Delta^{d_x} \times \Delta^{d_x}$ to itself, we have that $A_{\pi,\mathrm{cl}}^{t-i}$ maps $\Delta^{d_x} \times \Delta^{d_x}-\Delta^{d_x} \times \Delta^{d_x}$ to itself. Combining these three steps, we have that
$$
\| A_{compare}^t \begin{bmatrix}
v_1 \\
v_2
\end{bmatrix}-   A_{\pi,\mathrm{cl}}^t \begin{bmatrix}
v_1 \\
v_2
\end{bmatrix}\|_1\le 2\alpha d_x t.
$$
Because of our mixing time assumption on $A_{\pi,\mathrm{cl}}$, we have that for the same $x,s$,
$$
\sup_{p\in\Delta^{d_x},q\in\Delta^{d_x}}\|A_{compare}^{4\tau}(p,q)-(x,s)\|_1\le \frac{1}{16}+\frac{1}{16}\le \frac{1}{4}.
$$
As a result, $t^{\text{mix}}(A_{compare})\le 4\tau$. However, $\tau_A\le t^{\text{mix}}(A_{compare})$ by definition since the first half of $A_{compare}^{t}(p,q)$ is exactly $A^t p$ and therefore $\|A^{4\tau}p-x\|_1\le 1/4$, which leads to a contradiction.

Thus, we have
\begin{align*}
\|y_t-y_t(M_t;K_0)\|_1&\le \sum_{i=1}^{t-1}(1-\alpha)^{i-1}\|u_{t-i}-u_{t-i}(M_t;K_0)\|_1\\
&\le Cd_x^{7/2}H\beta\sum_{i=1}^{t-1}(1-\alpha)^{i-1}\cdot i\\
&\le Cd_x^{11/2}H\beta \tau^2.\\
\end{align*}
%\dhruv{(a minor error here, need to be a bit less lossy and not throw away the $\beta$ factor above)} 
By Lipschitz condition on $c_t$, we have
\begin{align*}
|c_t(y_t(M_t;K_0),u_t(M_t;K_0))-c_t(y_t,u_t)|&\le L(\|y_t-y_t(M_t;K_0)\|_1)\le C L d_x^{11/2} \beta\log^2(1/\beta)H\tau^2.
\end{align*}
\end{proof}

\subsection{Wrapping up the proof of Theorem~\ref{thm:po-main-ldc-simplex}}\label{sec:po-proof-nonmarkov}

Before proving Theorem~\ref{thm:po-main-ldc-simplex}, we establish that the loss functions $\ell_t$ used in $\NMGPCPOS$ are Lipschitz with respect to $\|\cdot\|_{\Sigma}$ in Lemma~\ref{lem:po-lt-lip-nonmarkov}, where $\|\cdot\|_{\Sigma}$ measures the $\ell_1$-norm of $M=M^{[0:H]}$ as a flattened vector in $\mathbb{R}^{d_x(d_y+d_x)(H+1)}$, formally given by
\begin{align*}
\|M_t\|_{\Sigma}:=\sum_{i=0}^{H}\sum_{j\in[d_x],k\in[d_y+d_x]}|(M_t^{[i]})_{jk}|.
\end{align*}

\begin{lemma} [Lipschitzness of $e_t$]
\label{lem:po-lt-lip-nonmarkov}
Suppose that a partially observable simplex LDS $\ML$ and $H\ge \tau>0$ are given, and $\Pi_{\tau}(\ML)$ is nonempty. Fix $K_0\in\mathbb{S}^{d_x,d_y}$. For each $t\in[T]$, the pseudo loss function $e_t$ (as defined on Line~\ref{line:nonmarkovet-loss} of \cref{alg:gpc-po-nonmarkov} is $O(Ld_x^{11/2}\tau)$-Lipschitz with respect to the norm $\|\cdot\|_{\Sigma}$ in $\MM_{d_x,d_y,H}^{+}$.  %\noah{did we say somewhere we're dropping the subscripts on $\mathcal{M}$?}
\end{lemma}

\begin{proof}
Since $e_t$ (Line~\ref{line:nonmarkovet-loss} of \cref{alg:gpc-po-nonmarkov}) is a composite of many functions, we analyze the Lipschitz condition of each of its components. First, we show that the $M\mapsto y_t(M;K_0)$ and $M\mapsto \tilde{u}_t(M;K_0)$ are Lipschitz in $M$ on $\MM_{d_x,d_y,H}^{+}$.
In particular, fix $M_1,M_2\in\MM_{d_x,d_y,H}^{+}$, we will show that 
\begin{align*}
\|y_t(M_1;K_0)-y_t(M_2;K_0)\|_1&\le O(d_x^{7/2}\tau) \|M_1-M_2\|_{\Sigma}, \\ 
\|\tilde{u}_t(M_1;K_0)-\tilde{u}_t(M_2;K_0)\|_1&\le 2\sqrt{d_x} \|M_1-M_2\|_{\Sigma}.
\end{align*}
Expanding $\tilde{u_t}$, we have
\begin{align*}
\|\tilde{u}_t(M_1;K_0)-\tilde{u}_t(M_2;K_0)\|_1&\le \sqrt{d_x}\|\tilde{u}_t(M_1;K_0)-\tilde{u}_t(M_2;K_0)\|_2\\
&\le \sqrt{d_x}\left\|(M_1^{[0]}-M_2^{[0]})y_t(K_0)+\sum_{i=1}^H\bar{\lambda}_{t,i}(M_1^{[i]}-M_2^{[i]})y_{t-i}(K_0)\right\|_1\\
&\le 2\sqrt{d_x} \max_{0\le i\le H} \|M_1^{[i]}-M_2^{[i]}\|_{1\rightarrow 1} \\
&\le 2\sqrt{d_x}\|M_1-M_2\|_{\Sigma},
\end{align*}
where the second inequality follows from the fact that $\ell_2$ norm is bounded by $\ell_1$ norm. By Lemma~\ref{cor:lip-control}, we have
\begin{align*}
\|u_t(M_1;K_0)-u_t(M_2;K_0)\|_1\le O(d_x^{7/2})\|M_1-M_2\|_{\Sigma}.
\end{align*}
Expanding $y_t$, we have that
\begin{align*}
\|y_t(M_1;K_0)-y_t(M_2;K_0)\|_1&=\left\|\sum_{i=1}^{t-1}\bar{\lambda}_{t,i}C[(1-\alpha)A]^{i-1}\alpha B(u_t(M_1)-u_t(M_2))\right\|_1.
\end{align*}
We consider two cases, depending on the mixing time $\tau_A:=\tmix(A)$ of $A$. 

\paragraph{Case 1:} $\tau_A\le 4\tau$. Let the stationary distribution of $A$ be denoted $p^{\star}\in\Delta^{d_x}$. Then by Lemma 22 in \citep{golowich2024online}, $\forall i\in\mathbb{N}$ and $v\in\mathbb{R}^d$ with $\langle \mathbbm{1}, v\rangle=0$, $\|A^{i}v\|_1\le 2^{-\lfloor i/\tau_A\rfloor}\|v\|_1\le 2^{-\lfloor i/4\tau\rfloor}\|v\|_1$. Then,
\begin{align*}
\|y_t(M_1)-y_t(M_2)\|_1&\le \sum_{i=1}^{t-1} \|CA^{i-1}B(u_t(M_1)-u_t(M_2))\|_1\\
&\le \sum_{i=1}^{t-1} 2^{-\lfloor (i-1)/4\tau\rfloor}\|u_t(M_1)-u_t(M_2)\|_1\\
&\le Cd_x^{7/2}\|M_1-M_2\|_{\Sigma}\sum_{i=1}^{t-1} 2^{1-\lfloor (i-1)/4\tau\rfloor}\\
&\le Cd_x^{7/2}\tau \|M_1-M_2\|_{\Sigma}. 
\end{align*}

\paragraph{Case 2:} $\tau_A>4\tau$. 
We already know from the previous discussion that $\alpha\ge 1/(256d_x \tau)$

We have
\begin{align*}
\|y_t(M_1)-y_t(M_2)\|_1&\le\sum_{i=1}^{t-1}(1-\alpha)^{i-1}\|(u_t(M_1)-u_t(M_2))\|_1\\
&\le Cd_x^{7/2}\|M_1-M_2\|_{\Sigma}\sum_{i=1}^{t-1}(1-1/(256d_x\tau))^{i-1}\\
&\le C d_x^{11/2}\tau\|M_1-M_2\|_{\Sigma}.
\end{align*}
Here we proved that $M\mapsto y_t(M;K_0)$ is $O(d_x^{11/2}\tau)$-Lipschitz, and $M\mapsto \tilde{u}_t(M;K_0)$ is $2\sqrt{d_x}$-Lipschitz. Immediately, $c_t^y(y_t(M;K_0))$ is $O(Ld_x^{11/2}\tau)$-Lipschitz on $\MM_{d_x,d_y,H}^{+}$. It suffices to show that
\begin{align*}
\pi_{\mathcal{K}}(\phi(\tilde{u}_t(M;K_0))) \cdot c_t^u(\phi^{-1}(\frac{\phi(\tilde{u}_t(M;K_0))}{\pi_{\mathcal{K}}(\phi(\tilde{u}_t(M;K_0))}))
\end{align*}
is Lipschitz on $\MM_{d_x,d_y,H}^{+}$. Since $\phi$ is $1$-Lipschitz over $\mathbb{R}^{d_y+d_x}$, we have that $M\mapsto \phi(\tilde{u}_t(M;K_0))$ is $2\sqrt{d_x}$-Lipschitz. By assumption $c_t^{u}$ is $L$-Lipschitz on $\Delta^{d_x}$, and $\phi^{-1}$ is $2$-Lipschitz over $\mathcal{K}$, we have $c_t^{u}\circ\phi^{-1}$ is $O(L)$-Lipschitz on $\mathcal{K}$. 

By Lemma~\ref{cor:lip-control}, 
\begin{align*}
u\mapsto \pi_{\mathcal{K}}(u)\cdot \left(c_t^u\circ\phi^{-1} \left(\frac{u}{\pi_{\mathcal{K}}(u)}\right)\right)
\end{align*}
is $O(Ld_x^{4})$-Lipschitz over $\mathcal{K}$, making $c_t^u(u_t(M;K_0))$ $O(Ld_x^{9/2})$-Lipschitz on $\MM_{d_x,d_y,H}^{+}$.  . Combining, we have that $e_t$ is $O(Ld_x^{11/2}\tau)$-Lipschitz on $\MM_{d_x,d_y,H}^{+}$.  
\end{proof}

\begin{proof} [Proof of \cref{thm:po-main-ldc-simplex}]
We are ready to prove \cref{thm:po-main-ldc-simplex}. It is crucial to check the condition in Lemma~\ref{lem:nonmarkovpo-mem-mismatch}. Define
\begin{align*}
\|M_t\|_{\star}:=\max_{0\le i\le H}\|M_{t}^{[i]}\|_{1\rightarrow 1}.
\end{align*}
It is easy to check that $\|M\|_{\Sigma}\ge \|M\|_{\star}$. With slight abuse of notation, let $\|M\|_2$ denote the $\ell_2$-norm of the flattened vector in $\mathbb{R}^{d_x(d_y+d_x)(H+1)}$ for $M\in\MM_{d_x,d_y,H}^{+}$. Let $\|\cdot\|_{\Sigma}^*$ denote the dual norm of $\|\cdot\|_{\Sigma}$. Lemma~\ref{lem:po-lt-lip-nonmarkov} implies that
\cref{alg:gpc-po-nonmarkov} guarantees
\begin{align*}
& \quad \|M_{t}-M_{t+1}\|_{\star}\le\|M_{t}-M_{t+1}\|_{\Sigma}\le_{(1)} \sqrt{d_x(d_y+d_x)H}\|M_{t}-M_{t+1}\|_2\\
&\le_{(2)} \eta \sqrt{d_x(d_y+d_x)H} \|\partial e_t(M_t)\|_{2}\le_{(3)}\eta \sqrt{d_x(d_y+d_x)H}\|\partial e_t(M_t)\|_{\Sigma}\\
&\le_{(4)} \eta (d_x(d_y+d_x)H)^{3/2} \|\partial e_t(M_t)\|_{\Sigma}^{*}\le_{(5)} \eta(d_x(d_y+d_x)H)^{3/2}Ld_x^{11/2}\tau\\
&\le 3\eta H^{5/2}Ld_x^{17/2}, 
\end{align*}
for some constant $C$, where $(1)$ and $(3)$ follow from $\ell_1$-norm-$\ell_2$-norm inequality that $\forall v\in\mathbb{R}^n$ $\|v\|_2\le \|v\|_1\le\sqrt{n}\|v\|_2$; $(2)$ follows from the update in Line~\ref{line:nonmarkovpo-update} of \cref{alg:gpc-po-nonmarkov};
$(4)$ follows from that $\|M\|_{\Sigma}\le d_xd_yH\|M\|_{\Sigma}^*$, and $(5)$ follows from Lemma~\ref{lem:po-lt-lip}, the assumption that $d_x\ge d_y$, and that $H\ge \tau$. 

Moreover, by the standard regret bound of OGD (Theorem 3.1 in \cite{hazan2016introduction}), we have
\begin{align*}
\sum_{t=1}^T e_t(M_t)-\min_{M\in\mathcal{M}}\sum_{t=1}^T e_t(M)\le O\left(\frac{D^2}{\eta}+(Ld_x^{11/2}\tau)^2\eta T\right), 
\end{align*}
where $D$ denotes the diameter of $\MM_{d_x,d_y,H}$ with respect to $\|\cdot\|_{\Sigma}$, which is bounded in this case by 
\begin{align*}
D=\sup_{M\in\mathcal{M}}\left\{ \sum_{i=0}^H\sum_{j\in[d_x],k\in[d_y+d_x]}|M^{[i]}_{jk}|\right\}\le 3(d_y+d_x)H.
\end{align*}
Take $\eta=1/(LH^{2}d_x^{6}\sqrt{T})$ and recall that $d_y\le d_x$, then we have
\begin{align*}
& \quad \sum_{t=1}^T c_t(y_t,u_t)-\min_{K^{\star}\in \Ksim_\tau(\ML)}\sum_{t=1}^T c_t(y_t(K^{\star}),u_t(K^{\star}))\\
&\le \sum_{t=1}^T c_t(y_t,u_t)-\sum_{t=1}^T c_t(y_t(M_t;K_0),u_t(M_t;K_0))+\sum_{t=1}^T c_t(y_t(M_t;K_0),u_t(M_t;K_0))- \sum_{t=1}^Te_t(M_t)\\
& \quad + \sum_{t=1}^Te_t(M_t)-\min_{M\in\mathcal{M}}\sum_{t=1}^T e_t(M)+1\\
&\le \sum_{t=1}^T e_t(M_t)-\min_{M\in\mathcal{M}}\sum_{t=1}^T e_t(M)+\tilde{O}\left(\eta H^{11/2}L^2d_x^{14} T\right)\\
&\le \tilde{O}\left(\frac{d_x^2H^2}{\eta}+(Ld_x^{11/2}\tau)^2\eta T+\eta H^{11/2}L^2d_x^{14} T\right)\\
&\le \tilde{O}\left(L\tau^{4}d_x^{8}\sqrt{T}\right),
\end{align*}
where the first inequality follows from Lemma~\ref{lem:po-approx-ldc} by taking $\eps=1$; the second inequality follows from Lemma~\ref{lem:nonmarkovpo-mem-mismatch} by taking $\beta=3\eta H^{5/2}Ld_x^{17/2}$ and since $c_t(y_t(M_t;K_0),u_t(M_t;K_0))\le e_t(M_t)$ from Observation~\ref{obs:ext}, and the third inequality follows from OGD regret guarantee. 
\end{proof}